\pdfoutput=1
\documentclass[a4paper,11pt]{amsart}
\usepackage{hyperref}
\usepackage{tikz}
\usepackage{amsmath}
\usepackage{amsfonts}
\usepackage{amssymb}
\usepackage{graphicx}
\usepackage[matrix,arrow]{xy}
\usepackage{epstopdf}
\usepackage{enumitem}
\usepackage{algorithm}
\usepackage[noend]{algorithmic}

\newlength\myindent
\newcommand{\suchthat}{\;\ifnum\currentgrouptype=16 \middle\fi|\;}

\def\Cox{\operatorname{Cox}}

\def\mov{\operatorname{Mov}}
\def\Mov{\mov}
\def\aa{\ensuremath{\mathfrak{a}}}

\def\KK{\mathbb{K}}
\def\ZZ{\mathbb{Z}}
\def\QQ{\mathbb{Q}}

\def\QQQ{\QQ_{\geq 0}}
\def\KT#1{\KK[T_1,\ldots,T_{#1}]}
\def\<{\langle}
\def\>{\rangle}

\def\cone{{\rm cone}}

\newcommand{\seite}{\preceq}

\theoremstyle{plain}

\newtheorem{theorem}{Theorem}[section]
\newtheorem{lemma}[theorem]{Lemma}
\newtheorem{proposition}[theorem]{Proposition}
\newtheorem{corollary}[theorem]{Corollary}
\theoremstyle{definition}
\newtheorem{definition}[theorem]{Definition}
\newtheorem{example}[theorem]{Example}
\newtheorem{algo}[theorem]{Algorithm}

\newtheorem{remark}[theorem]{Remark}
\newtheorem{construction}[theorem]{Construction}

\subjclass[2010]{
Primary 14L24; Secondary 13A50, 14Q99, 13P10, 68W10.
}
\keywords{Geometric invariant theory, group action, GIT-fan, parallel computation, Mori dream spaces}

\begin{document}
\title[Computing GIT-fans with symmetry]{Computing GIT-fans with symmetry and \\the Mori chamber decomposition of $\overline{M}_{0,6}$}
\author[J.~B\"ohm]{Janko~B\"ohm}
\address{Department of Mathematics\\
University of Kaiserslautern\\
Erwin-Schr\"odinger-Str.\\
67663 Kaiserslautern\\
Germany}
\email{boehm@mathematik.uni-kl.de}
\author[S.~Keicher]{Simon~Keicher}
\address{
Departamento de Matematica\\
Facultad de Ciencias Fisicas y Matematicas\\
Universidad de Concepci{\'o}n\\ Casilla 160-C, Concepci{\'o}n, Chile
}
\email{keicher@mail.mathematik.uni-tuebingen.de}
\author[Y.~Ren]{Yue~Ren}
   \address{Ben-Gurion University of the Negev\\
     Department of Mathematics\\
     P.O.B. 653\\
     8410501 Be'er Sheva, Israel
     }
   \email{reny@post.bgu.ac.il}
\thanks{The authors acknowledge support of the DFG SPP 1489. The second author was supported by proyecto FONDECYT postdoctorado N.~3160016.}

\begin{abstract}
We propose an algorithm to compute the GIT-fan for torus actions on affine varieties with symmetries. The algorithm
combines computational techniques from commutative algebra, convex geometry and group
theory.
We have implemented our algorithm in the \textsc{Singular} library
\textsc{gitfan.lib}. Using our implementation, we compute the Mori chamber
decomposition of~$\operatorname{Mov}(\overline{M}_{0,6})$.

\end{abstract}
\maketitle

\section{Introduction}

Dolgachev/Hu~\cite{DolgachevHu} and Thaddeus~\cite{Tha} assigned to an
algebraic variety with the action of an algebraic group the \emph{GIT-fan\/},
a polyhedral fan enumerating the GIT-quotients in the sense of
Mumford~\cite{Mumford}. The case of the action of an algebraic torus $H$ on an
affine variety $X$ has been treated by Berchtold/Hausen~\cite{BeHa}. Based on
their construction, an algorithm to compute the GIT-fan in this setting has
been proposed in~\cite{Ke}. Note that this setting is essential for many applications, since the torus case can be used to
investigate the GIT-variation of the action of a connected reductive group $G$, see~\cite{AH}.

In many important examples, $X$ is symmetric under the action of a finite group which either is known directly from its geometry or can be computed, e.g., using~\cite{HaKeWo}. A prominent instance is the Deligne-Mumford compactification
$\overline{M}_{0,6}$ of the moduli space of $6$-pointed stable curves of genus
zero, which has a natural action of the symmetric group~$S_{6}$.
In this paper, we  address two main problems:
\begin{itemize}[leftmargin=*]
 \item to develop an efficient algorithm computing GIT-fans, which makes use of symmetries, and
 \item to determine the Mori chamber decomposition of the cone of movable divisor classes of~$\overline M_{0,6}$.
\end{itemize}

We first describe an algorithm that determines the
GIT-fan by computing exactly one representative in each orbit of maximal cones. Each cone is represented by a single integer.  
The algorithm
relies on Gr\"{o}bner basis techniques, convex geometry and
actions of finite symmetry groups. It demonstrates the strength of cross-boarder methods in computer algebra, and the efficiency of the algorithms implemented in all involved areas. The algorithm is also suitable for parallel computations.
We provide an implementation in the
library \textsc{gitfan.lib}~\cite{gitfanlib} for the 
computer
algebra system \textsc{Singular}\footnote{The library is available in
the current development version and will be part of the next release.}~\cite{singular}. The implementation is an interesting use case for the current efforts to connect different Open Source computer algebra systems, see~\cite[Sec. 2.4]{BDSR}.

We then turn to $\overline{M}_{0,6}$, which is known to be a
\emph{Mori dream space\/}, that is, its Cox ring $\operatorname{Cox}(\overline
{M}_{0,6})$ is finitely generated, see~\cite{HuKe}.
Castravet~\cite{Ca} has
determined generators for $\operatorname{Cox}(\overline{M}_{0,6})$ and Bernal
Guill\'{e}n~\cite{Ber} the relations as well as an explicit description of the symmetry group action.
An interesting
open problem is the computation of the \emph{Mori chamber
decomposition\/} of the cone of movable divisor classes $\operatorname{Mov}%
(\overline{M}_{0,6})\subseteq\operatorname{Eff}(\overline{M}_{0,6})$, 
see~\cite{HaTs} for a description of these cones in terms of generators. This fan
is the decomposition of $\operatorname{Mov}(\overline{M}_{0,6})$ into chambers
of the GIT-fan of the action of the characteristic torus on its total
coordinate space; it characterizes the birational geometry of~$\overline
{M}_{0,6}$. In Section~\ref{M06}, we solve the mentioned problem
and obtain the following result.

\begin{theorem}
\label{thm:M06gitfan}
The Mori chamber decomposition of $\operatorname{Mov}(\overline{M}_{0,6})$ is a (pure)
$16$-dimensional fan with $176\,512\,180$ maximal cones and $296\,387$ rays.
The
set of maximal cones decomposes into $249\,604$ orbits of $S_{6}$, the set of rays into $9\,218$ orbits. For the maximal cones, the number of
orbits of a given cardinality is as follows:%
\[
{\footnotesize
\begin{tabular}
[c]{l|lllllllllllllll}%
\text{cardinality} & $1$ & $6$ & $10$ & $15$ & $20$ & $30$ & $45$ & $60$ &
$72$ & $90$ & $120$ & $180$ & $240$ & $360$ & $720$\\\hline
\text{no.~of orbits} & $1$ & $1$ & $1$ & $4$ & $1$ & $1$ & $9$ & $27$ & $4$ &
$46$ & $32$ & $488$ & $4$ & $7934$ & $241051$%
\end{tabular}
}
\]

\end{theorem}

The complete data of the fan including vectors in the relative interior of
each maximal cone is available at~\cite{gitfandata}.

This problem is computationally challenging both due to the complexity of the
input, the resulting fan and the intermediate data to be handled in the course of the
computation. Hence, aside from the theoretical importance, it is a meaningful
benchmark for the symmetric GIT-fan algorithm.

This paper is structured as follows.
In Section~\ref{sec computing GIT-fans},
we introduce our notation and recall the algorithm of \cite{Ke} for
computing GIT-fans; this will be our starting-point for developing an
algorithm computing GIT-fans with symmetries. In Section~\ref{sec monomial containment}, we present an efficient test for monomial
containment. The test is a key ingredient to the GIT-fan algorithm, but is also relevant in a broader sense, for example, for computing tropical varieties.  We give timings, which illustrate that our method is
outperforming the known methods by far.
In Section~\ref{sec:symmetricGITfan},
we describe the symmetric GIT-fan algorithm as well as implementation
details.
It is followed by two explicit example computations in
Section~\ref{examples}.
Finally, in Section~\ref{M06}, we apply this algorithm to
compute the Mori chamber decomposition of the moving cone of~$\overline
{M}_{0,6}$.

{\em Acknowledgements.} We would like to thank J\"urgen Hausen for turning our  interest to the subject. We also thank Hans Sch\"onemann for helpful discussions, and the \textsc{Singular}-group of the University of Kaiserslautern for providing resources for the computation. We further thank Antonio Laface and Diane Maclagan for helpful and interesting discussions.

\section{Computing GIT-Fans}

\label{sec computing GIT-fans}

In this section, we recall from~\cite{Ke, ArDeHaLa, BeHa} the setting and an algorithm to compute GIT-fans. Moreover, we fix our notation. This section serves as a starting point for our advanced algorithm described in the subsequent sections.

We work in the following setting. Let $\mathbb{K}$ be an algebraically closed
field of characteristic zero. Consider an affine variety $X\subseteq
\mathbb{K}^{r}$ over $\mathbb{K}$, acted on effectively by an algebraic torus
$H:=(\mathbb{K}^{*})^{k}$ where $k\in\mathbb{Z}_{\geq1}$. We assume that $X$
is given as a zero set $X=V(\mathfrak{a})\subseteq\mathbb{K}^{r}$ of a
monomial-free ideal $\mathfrak{a}\subseteq\mathbb{K}[T_{1},\ldots,T_{r}]$.
Note that the $H$-action on $X$ can be encoded in an integral matrix
$Q\in\mathbb{Z}^{k\times r}$ of full rank. Denoting the columns of $Q$ by
$q_{1},\ldots,q_{r}$, the ideal $\mathfrak{a}\subseteq\mathbb{K}[T_{1}%
,\ldots,T_{r}]$ is homogeneous with respect to the $\mathbb{Z}^{k}$-grading
\[
\deg(T_{1})\,=\,q_{1},\,\quad\ldots,\quad\deg
(T_{r})\,=\,q_{r}.
\]
The \emph{GIT-fan\/} of the $H$-action on $X$ is a pure, $k$-dimensional
polyhedral fan $\Lambda(\mathfrak{a},Q)$ in $\mathbb{Q}^{k}$ with support
$\mathrm{cone}(q_1,\ldots,q_r)$. The cones of the GIT-fan $\Lambda(\mathfrak{a} ,Q)$ are
called \emph{GIT-cones\/}. They enumerate the sets of semistable points
$X^{\mathrm{ss}}(w)\subseteq X$ that admit a good quotient by $H$ with
quasi-projective quotient space $X^{\mathrm{ss}}(w){/\!\!/} H$ and that
satisfy a certain maximality condition, see~\cite[Section~1.4]{ArDeHaLa}
and~\cite{BeHa} for details.

The GIT-fan can be computed by Algorithm \ref{algo:gitfan} from~\cite{Ke}. To describe this approach, we
use the following notation. Given an $r$-tuple $z=(z_{1},\ldots,z_{r})$ and a
face $\gamma_{0}\preceq\gamma$ of the positive orthant $\gamma:= \mathbb{Q}^{r}_{\geq0}$, define the
restriction $z_{\gamma_{0}}$ via
\[
(z_{\gamma_{0}})_{i}\ :=\
\begin{cases}
z_{i}, & e_{i}\in\gamma_{0},\\
0, & e_{i}\notin\gamma_{0},
\end{cases}
\qquad1\leq i\leq r.
\]
If the ideal $\mathfrak{a}$ is generated by $g_{1},\ldots,g_{s} \in \KT{r}$ we write
$\mathfrak{a} _{\gamma_{0}}\subseteq\mathbb{K}[T_{\gamma_{0}}]$ for the ideal
generated by $g_{1}(T_{\gamma_{0}}),\ldots,g_{s}(T_{\gamma_{0}})$, where
$T=(T_{1},\ldots,T_{r})$. We call a face $\gamma_{0}\preceq\gamma$ an
\emph{$\mathfrak{a}$-face} if the corresponding torus orbit meets the variety,
that is,
$$
X\cap\mathbb{T}_{\gamma_{0}}\,\ne\,\emptyset\qquad\text{where}\qquad \mathbb{T}_{\gamma
_{0}}\,:=\,(\mathbb{K}^{*})^{r}\cdot(1,\ldots,1)_{\gamma_{0}}.$$
Projecting an
$\mathfrak{a} $-face $\gamma_{0}\preceq\gamma$ to $\mathbb{Q}^{k}$ via $Q$ yields the \emph{orbit cone}
$Q(\gamma_{0})\subseteq\mathbb{Q}^{k}$. Writing $\Omega$ for the (finite) set
of all orbit cones, the \emph{GIT-cones} are the polyhedral cones
\[
\lambda_{\Omega}(w) \ :=\ \bigcap_{\substack{\vartheta\in\Omega\\w\in
\vartheta^{\circ}}} \vartheta\ \subseteq\ \mathbb{Q}^{k} \qquad\text{where}%
\ w\in Q(\gamma).
\]

In the following, by an \emph{interior facet} of a full-dimensional cone
$\lambda\subseteq Q(\gamma)$, we mean a facet $\eta\preceq\lambda$ such that
$\eta$ meets the relative interior $Q(\gamma)^{\circ}$ non-trivially.
Moreover, we denote by $\ominus$ the symmetric difference in the first component, that is, given two subsets
$A,B\subseteq M\times N$ of sets $M$ and~$N$ we set
\[
A\ominus B\,:=\, \{(\eta,\lambda)\in A\cup B \mid\eta\in\pi_M(A) \text{ xor }
\eta\in\pi_M(B)\},
\]
where $\pi_M\colon M\times N\rightarrow M$ is the projection onto the first component. We are now ready to state the algorithm to compute the GIT-fan~$\Lambda(\aa,Q)$.

\begin{algo}[Compute the GIT-fan]
\label{algo:gitfan}\
\begin{algorithmic}[1]
\REQUIRE{An ideal $\aa\subseteq \KK[T_1,\ldots,T_{r}]$
and a matrix $Q\in\mathbb{Z}^{k\times r}$ of full rank  such that $\aa$ is homogeneous with respect to the multigrading given by $Q$.}
\ENSURE{The set of maximal cones of $\Lambda(\aa, Q)$.}
\STATE $\mathcal A := \{\hspace{1mm}\}$
\FORALL {faces $\gamma_0\seite \QQQ^r$}
\label{afacesloop1}
\IF{
$\gamma_0$ is an $\aa$-face as verified by Algorithm~\ref{algo:isaface}
}
\STATE $\mathcal A := \mathcal A \cup \{\gamma_0\}$
\ENDIF
\ENDFOR
\STATE $\Omega:=\left\{Q (\gamma_0) \mid \gamma_0 \in  \mathcal A\right\}$
\STATE Choose a vector $w_0\in Q(\gamma)^\circ$ such that $\dim(\lambda_\Omega(w_0))=k$.
\STATE Initialize $\mathcal C := \{\lambda(w_0)\}$ and $\mathcal{F}:=\{(\tau, \lambda_\Omega(w_0))\mid \tau\seite\lambda(w_0) \text{ interior facet}\}$.
\WHILE {there is $(\eta,\lambda)\in\mathcal{F}$}
\label{fan trav}
\STATE Find $w \in Q(\gamma)^\circ$ such that $\lambda_\Omega(w) \cap \lambda=\eta$\label{line:neighbour}.
\STATE $\mathcal C := \mathcal C \cup \{\lambda_\Omega(w)\}$
\STATE $\mathcal F := \mathcal F \ominus \{(\tau,\lambda_\Omega(w))\mid\tau\seite\lambda_\Omega(w) \text{ interior facet}\}$
\ENDWHILE
\RETURN $\mathcal C$
\end{algorithmic}
\end{algo}

\begin{algo}[$\aa$-face test]
\label{algo:isaface}\
\begin{algorithmic}[1]
\REQUIRE{Generators $g_1,\ldots,g_s$ for an ideal $\aa\subseteq \KK[T_1,\ldots,T_r]$, and a face $\gamma_0\seite \gamma$.}
\ENSURE{\texttt{true} if $\gamma_0$ is an $\aa$-face, \texttt{false} else.}
\RETURN {$1\notin\aa_{\gamma_0}:(\prod_{e_i\in\gamma_0}T_i)^\infty$}
\end{algorithmic}
\end{algo}

Algorithm \ref{algo:gitfan} will be our starting-point for developing an
efficient method for computing GIT-fans with symmetry in
Section~\ref{sec:symmetricGITfan}.
Algorithm \ref{algo:isaface} is an ad-hoc algorithm for determining $\aa$-faces.
How to improve its performance will be discussed in the next section.

\begin{remark}\
\begin{enumerate}[leftmargin=*]
\item Note that in Algorithm~\ref{algo:isaface}, instead of computing the
saturation, one can also perform the radical membership test $\prod_{e_i\in\gamma_0}T_i\in \sqrt{\aa_{\gamma_0}}$. Both approaches require
Gr\"{o}bner basis computations.
\item In Line~\ref{line:neighbour} of Algorithm~\ref{algo:gitfan}, we find $w$
by adding an appropriate small positive multiple of an outer normal of $\lambda$ at $\eta$
to a vector in the relative interior~$\eta^{\circ}$.
\end{enumerate}
\end{remark}

\section{Closure computation\label{sec monomial containment}}

The first bottle-neck in Algorithm \ref{algo:gitfan} is the computation of the
$\mathfrak{a}$-faces using Algorithm~\ref{algo:isaface}. In this section,
we present a fast algorithm for the saturation of an ideal at a union of
coordinate hyperplanes. Geometrically, this process corresponds to computing
the closure $\overline{X}\subseteq\mathbb{K}^{n}$ of a given subvariety
$X\subseteq(\mathbb{K}^{*})^{n}$. In particular, this algorithm gives an
efficient monomial containment test, which is superior to the standard approaches using the
Rabinowitsch trick or saturation.
We first present the algorithm and then illustrate its efficiency by providing a series of timings.

In this section, we have no assumptions on the field $\mathbb{K}$.
Consider an ideal
$I\subseteq R:=\mathbb{K}[Y_{1},\ldots,Y_{n}]$. We describe an algorithm for
computing $I:\left(  Y_{1}\cdots Y_{m}\right)  ^{\infty}$, where $m\leq n$.
A~key ingredient is the following generalization of~\cite[Lemma~12.1]%
{Stu}. Denote by $\operatorname{LM}_{>}(f)$ the leading monomial of a
polynomial $f\in R$ with respect to a monomial ordering~$>$.

\begin{proposition}\label{prop saturierung}
Let $>$ be a monomial ordering on $R$ and $\mathcal{G}$ a Gr\"{o}bner basis of $I$.
Suppose that for all $f\in\mathcal{G}$ we have
\[ Y_{m}\mid f\quad\Longleftrightarrow\quad Y_{m}\mid\operatorname{LM}_{>}(f). \]
Then
\[
\left\{  f\in\mathcal{G}\mid Y_{m}\text{ does not divide }f\right\}
\cup\left\{  \frac{f}{Y_{m}}\;\ifnum\currentgrouptype=16 \middle\fi|\;
f\in\mathcal{G}\text{, }Y_{m}\text{ divides }f\right\}  
\]
is a Gr\"{o}bner basis of the ideal quotient $I:Y_{m}$, and
\[
\left\{  \frac{f}{Y_{m}^{i}}\;\ifnum\currentgrouptype=16 \middle\fi|\;
f\in\mathcal{G}\text{ and }i\geq0\text{ maximal such that }Y_{m}^{i}\mid
f\right\}\]
is a Gr\"{o}bner basis for the saturated ideal $I:Y_{m}^{\infty}$.
\end{proposition}

\begin{proof}
Immediate generalization of the proof of~\cite[Lemma~12.1]{Stu}.
\end{proof}

\begin{remark}
\label{rem:monomialOrderings}
Consider the setting of Proposition~\ref{prop saturierung}.
\begin{enumerate}[leftmargin=*]
\item If $I$ is weighted homogeneous with respect to the weight
vector $w\in\mathbb{Q}^{n}$ with $w_{i}>0$ for all $i$, then we can use a
$w$-weighted degree ordering $>_{w}$ with a negative reverse lexicographical
tie-breaker ordering%
\[
\qquad
Y^{\alpha}>_{\text{rs}}Y^{\beta}\;\;:\Longleftrightarrow\;\;
\alpha_{n}=\beta_{n},\;\ldots,\;\alpha_{i+1}=\beta_{i+1} \text{ and }\alpha_{i}<\beta_{i} \text{ for some } n\geq i\geq 1.\]
\item In particular, if $\mathcal{G}$ is homogeneous with respect to the
standard grading, then we can use the graded reverse lexicographic term
ordering, see \cite[Lemma 12.1]{Stu}.

\item Proposition~\ref{prop saturierung} is also correct in the setting of local orderings
and standard bases. In this case, the assumption of the proposition is always
satisfied for the negative reverse lexicographical ordering.
\end{enumerate}
\end{remark}

The following algorithm computes the saturation of a weighted homogeneous ideal at the product of the first $m$ variables using Proposition~\ref{prop saturierung} and a modified Buchberger's algorithm. The modification lowers the degrees of the computed Gr\"obner basis elements, thereby leading to an earlier stabilization of intermediate leading ideals and, hence, earlier termination of the algorithm.

\begin{algo}[Saturation at a product of variables]
\label{algo:sat}\
\begin{algorithmic}[1]
\REQUIRE A set of $w$-homogeneous generators $\mathcal{G}\subseteq I$ of an ideal $I\subseteq R=\mathbb{K}[Y_{1},\ldots,Y_{n}]$ for some weight vector $w\in\ZZ^n_{>0}$, an integer $m\leq n$.\
\ENSURE A Gr\"obner basis for the saturation $I:(Y_1\cdots Y_m)^\infty$ with respect to the $w$-weighted negative reverse lexicographical ordering as in Remark \ref{rem:monomialOrderings}.
\FOR{$i=1,\ldots,m$}
\STATE Let $>_w$ be the $w$-weighted degree ordering with the negative reverse lexicographical tie-breaker $>_{\text{rs}}$ such that
\[ Y_1 >_{\text{rs}} \ldots >_{\text{rs}} Y_{i-1} >_{\text{rs}} Y_{i+1} >_{\text{rs}} \ldots >_{\text{rs}} Y_n >_{\text{rs}} Y_i.\vspace{-0mm} \]
\emph{Apply Buchberger's algorithm to $\mathcal{G}$ with the following modification:}
\REPEAT
\STATE $\mathcal{H}:=\mathcal{G}$
\FORALL {$f,g\in \mathcal{H}$}
\STATE $r:=\operatorname*{NF}_{>_w}(\operatorname*{spoly}_{>_w}(f,g),\mathcal{H})$
\IF {$r\neq0$}
\STATE $r:=r/(Y_1^{\alpha_1}\cdots Y_m^{\alpha_m})$, where $\alpha_j$ is maximal such that $Y_j^{\alpha_j} \mid r$.
\STATE $\mathcal{G}:=\mathcal{G}\cup\{r\}$\label{line add element}
\ENDIF
\ENDFOR
\UNTIL{$\mathcal{G}=\mathcal{H}$}
\ENDFOR
\RETURN{$\mathcal{G}$}
\end{algorithmic}
\end{algo}

\begin{proof}

Termination follows by the Noetherian property since in Line
\ref{line add element} the lead ideal of $\left\langle \mathcal{G}
\right\rangle $ strictly increases.

Denote by $\mathcal{G}_{i}$ the Gr\"obner basis after step~$i$ and by $I_{i}$ the
ideal generated by it. Because none of the elements of $\mathcal{G}_{i}$ is divisible by $Y_{i}$
and due to the choice of the monomial ordering, Proposition~\ref{prop saturierung} implies that
$I_{i}$ is saturated with respect to $Y_{i}$. Therefore, we have%

\[
\underbrace{I:Y_{1}^{\infty}}_{\subseteq I_{2}}:Y_{2}^{\infty}:\ldots
:Y_{m}^{\infty}\subseteq\underbrace{I_{2}:Y_{2}^{\infty}}_{\subseteq I_{3}%
}:\ldots:Y_{m}^{\infty}\subseteq\ldots\subseteq\underbrace{I_{m-1}%
:Y_{m}^{\infty}}_{\subseteq I_{m}} \subseteq I_{m}.
\]
The claim follows from the fact that for all $1\leq i\leq m$ we have
\[
I:Y_{1}^{\infty}:\ldots:Y_{i}^{\infty}\,\subseteq\, I_i=\langle\mathcal{G}_{i}
\rangle\,\subseteq\, I:(Y_{1}\cdots Y_{m})^{\infty}.\qedhere
\]
\end{proof}

With regard to timings, we compare Algorithm \ref{algo:sat} as implemented in
the \textsc{Singular} library \textsc{gitfan.lib} with other standard methods
for computing saturations. Here we consider the ad-hoc algorithm given by
Proposition~\ref{prop saturierung}, the computation of saturations by iterated ideal
quotients (SAT). We also give timings
for the use of the trick of Rabinowitsch to determine monomial containment
(RA). All algorithms are implemented in \textsc{Singular}. To improve the performance, the implementations of Algorithm \ref{algo:sat} and Proposition~\ref{prop saturierung} use a parallel computation strategy to heuristically determine an ordering of the variables for the iterated saturation. All other algorithms are implemented in a sequential way. The timings are in seconds on an AMD Opteron 6174 
machine with 48 cores,
$2.2$~GHz, and $128$~GB of RAM.

As an example, we consider the ideal $\aa\subseteq R=\mathbb{Q}[y_{1234},\ldots,z_{156}]$  obtained from Algorithm~\ref{algo:coxm06}. It has $225$ generators in $40$ variables.
Timings for the ideal $\aa_J:=\aa_{\mathrm{cone}(e_{j}\mid
j\in J)}$, as defined in Section \ref{sec computing GIT-fans},
 are given in Table \ref{tab timings}. In the cases marked by a star, the
computation did not finish within one day.

 \begin{table}[H]
\begin{center}%

\begin{tabular}
[c]{c|c|c|c|c|c|c}%
$\{1,\ldots,40 \}\backslash J$ & $40-|J|$ & $\aa$-face & Alg. \ref{algo:sat} & Prop. \ref{prop saturierung} & SAT & RA
\ \\\hline
$\{3,4,5,7,\ldots,15\}$&28& no & 1     & 761 & 44 & 70 \\
$\{9,11,12,13,15\}$&$35$& no   & 1     & 57200 & 13400 & 40300 \\
$\{11,12,13,15\}$&$36$& no     & 1     & 44100 & 9140  & 38500  \\
$\{9,11,14,15\}$&$36$& yes     & 48    & $\ast$ & $\ast$  & $\ast$ \\
$\{9,11,15\}$&$37$& yes        & 920  & $\ast$ & $\ast$ & $\ast$  \\
$\{9,11,13\}$&$37$& no         & 1     & 31400 & 7610 & 24300
\end{tabular}

\end{center}
 \medskip
 \caption{Timings for computing the closure.}%
 \label{tab timings}
 \end{table}

\section{Computing GIT-Fans with Symmetry\label{sec:symmetricGITfan}}

As in Section~\ref{sec computing GIT-fans}, we consider an ideal $\aa\subseteq \KT{r}$ that is homogeneous with respect to the $\ZZ^k$-grading on $\KT{r}$
given by assigning to $T_i$  the $i$-th column of an integral $(k\times r)$-matrix $Q$ as its degree;
this encodes the action of $H=(\KK^*)^k$ on $X=V(\aa)\subseteq \KK^r$.
In this section,
we provide an efficient algorithm to compute the GIT-fan
$\Lambda(\aa,Q)$ if symmetries of the input are known.
By symmetries, we mean the following.

\begin{definition}
\label{def:symmgrp}
A \emph{symmetry
group} of the action of $H$ on $X$ is a subgroup $G$ of the symmetric group
$S_{r}$
such that there are group actions%
\[%
\begin{array}
[c]{rclcllcl}%
G & \times & \mathbb{K}[T_{1},\ldots,T_{r}] & \rightarrow & \mathbb{K}%
[T_{1},\ldots,T_{r}], & (\sigma,T_{j}) & \mapsto & \sigma(T_{j}%
)\ :=\ c_{\sigma,j}\cdot T_{\sigma(j)}\\
G & \times & \mathbb{Q}^{r} & \rightarrow & \mathbb{Q}^{r}, & (\sigma,e_{j}) &
\mapsto & \sigma(e_{j})\ :=\ e_{\sigma(j)}\\
G & \times & \mathbb{Q}^{k} & \rightarrow & \mathbb{Q}^{k}, & (\sigma,v) &
\mapsto & A_{\sigma}\cdot v
\end{array}
\]
with $A_{\sigma}\in\mathrm{\operatorname{GL}}(k,\mathbb{Q})$ and $c_{\sigma
}\in(\KK^*)^{r}$ such that $G\cdot\mathfrak{a}=\mathfrak{a}$ holds
and for
each $\sigma\in G$ the following diagram is commutative:
\[
\xymatrix{
\QQ^r\ar[d]_Q \ar[rr]^{e_i\,\mapsto\, e_{\sigma(j)}}
&&
\QQ^r \ar[d]^Q
\\
\QQ^k \ar[rr]^{A_\sigma} && \QQ^k
}
\]
Note that the existence of such a linear map $A_{\sigma}$ is
equivalent to
$\sigma \ker(Q)$ being a subset of the kernel~$\ker(Q)$. Note also that for the graded components $\aa_w$, where $w\in \ZZ^k$,  we have  $\sigma\cdot \aa_w = \aa_{A_\sigma w}$ for all $\sigma \in G$.
\end{definition}

\begin{remark}
Symmetries of a homogeneous ideal as in Definition~\ref{def:symmgrp} can be computed with the methods of~\cite{HaKeWo}.
\end{remark}

From now on, we fix a symmetry group $G$ for the $H$-action on $X\subseteq \KK^r$.
Our goal is to modify Algorithm~\ref{algo:gitfan} such that it can exploit the symmetries given by~$G$.

The first improvement to Algorithm~\ref{algo:gitfan} concerns the representation
of GIT-cones: we will encode them in a binary number,
such that the representation is compatible with the group action. This binary number, in turn, can be
interpreted as an integer. This yields a total ordering on the set of
GIT-cones. In conjunction with the easily computable representation, this
allows for an efficient test for membership of a given GIT-cone in a set of GIT-cones. Such a
representation is also called a {\em perfect hash function\/}.

\begin{construction}[Encoding GIT-cones as integers]
\label{con:hash}
Let the setting be as above, i.e., denote by $\Omega$ the set of orbit cones and by $\Lambda(\aa,Q)$ the GIT-fan.
Consider the map $h_\Omega$
and the action of $G$ on $\left\{  0,1\right\}^{\Omega}$
given by
\begin{gather*}
h_{\Omega}\colon\Lambda(\mathfrak{a},Q)\,\to\,\left\{  0,1\right\}
^{\Omega}\text{,\quad\quad}\lambda\,\mapsto\,\left[
\begingroup
\footnotesize
\begin{tabular}
[c]{l}%
$\Omega\rightarrow\left\{  0,1\right\}  $\\
$\vartheta\mapsto%
\begin{cases}
1 & \lambda\subseteq\vartheta\\
0 & \lambda\nsubseteqq\vartheta
\end{cases}
$%
\end{tabular}
\endgroup
\right]
,\\
G\times\left\{  0,1\right\}  ^{\Omega}\,\to\,\left\{  0,1\right\}
^{\Omega}\text{,\quad\quad}(g,b)\,\mapsto\,\left[
\begingroup
\footnotesize
\begin{tabular}
[c]{l}%
$\Omega\rightarrow\left\{  0,1\right\}  $\\
$\vartheta\mapsto b(g^{-1}\cdot\vartheta)$%
\end{tabular}
\endgroup
\right].
\end{gather*}
Then the map $h_{\Omega}$ is injective. Moreover, for all $g\in G$ and GIT-cones $\lambda\in\Lambda(\mathfrak{a},Q)$, we have
$g\cdot h_{\Omega}(\lambda)\ =\ h_{\Omega}(g\cdot\lambda)$.
\end{construction}

\begin{proof}
Any element of $\Lambda(\mathfrak{a},Q)$ is of the form $\lambda_{\Omega
}(w)\ $where$\ w\in Q(\gamma)$, that is, it is the intersection of all elements of
$\Omega$ that contain $w$. This implies that $h_{\Omega}$ is injective.
Compatibility with the group action follows immediately, since%
\[
g\cdot h_{\Omega}(\lambda)=\left[
\begingroup
\footnotesize
\begin{tabular}
[c]{l}%
$\Omega\rightarrow\left\{  0,1\right\}  $\\
$\vartheta\mapsto%
\begin{cases}
1 & \lambda\subseteq g^{-1}\cdot\vartheta\\
0 & \lambda\nsubseteqq g^{-1}\cdot\vartheta
\end{cases}
$%
\end{tabular}
\endgroup
\right]
\,=\,
h_{\Omega}(g\cdot\lambda)\text{.}\qedhere
\]
\end{proof}

\begin{remark}
Consider Construction~\ref{con:hash}.
\begin{enumerate}[leftmargin=*]
 \item
With respect to the practical implementation, recall that any binary number
determines a unique integer via its $2$-adic representation. We test
membership in a given set of GIT-cones by a binary search in an ordered list
of integers representing the set. To insert elements we use insertion sort.
\item
Our approach is more efficient than representing maximal cones in
terms of the sum of the rays, since, in the GIT-fan algorithm, cones are
naturally given in their representation in terms of half-spaces and hyperplanes, and computation of the
representation in terms of rays by double description is expensive. Note also, that in
our representation, the group action is given by permutation of bits, whereas
the action on the sum of rays requires a matrix multiplication.
\end{enumerate}
\end{remark}

We now state our refined, symmetric GIT-fan Algorithm~\ref{algo:gitfanWithSymmetry}.
When computing the $\aa$-faces, the algorithm considers a distinct set of representatives  of the orbits of the faces of $\gamma$ with regard
to the action of the symmetry group. For the individual tests, the efficient saturation computation as described in Algorithm~\ref{algo:sat} is applied.
For computing the GIT-cones of maximal dimension, the algorithm works with a reduced set of orbit cones.
With regard to the symmetry group action, it computes exactly one cone per orbit of GIT-cones, traversing facets only if necessary. The cones are represented via Construction~\ref{con:hash}.
In the following, we write $\Omega(k)$ for the full-dimensional orbit cones.

\pagebreak[3]
\begin{algo}[Computing symmetric GIT-fans]
\label{algo:gitfanWithSymmetry}\
\begin{algorithmic}[1]
\REQUIRE{A monomial-free  ideal $\aa\subseteq \KK[T_1,\ldots,T_{r}]$
and a matrix $Q\in\mathbb{Z}^{k\times r}$ of full rank such that $\aa$ is homogeneous with respect to the multigrading given by $Q$,
and a symmetry group $G$ of the action of $H=(\KK^*)^k$ on $X=V(\aa)$ given by $Q$.}
\ENSURE{A system of distinct representatives of the orbits of the $G$-action on $\Lambda(\aa, Q)(k)$.}
\STATE $\mathcal A := \{\hspace{1mm}\}$\label{line:step1}
\STATE $\mathcal{S}:=$ system of distinct representatives of the orbits of the $G$-action on $\operatorname{faces}(\gamma)$
\FORALL{$\gamma_0\in\mathcal{S}$}\label{afacesloop1}
\IF{$\gamma_0$ is an $\aa$-face  as verified by Algorithm~\ref{algo:isaface} using Algorithm~\ref{algo:sat}}
\STATE $\mathcal A := \mathcal A \cup \{\gamma_0\}$
\label{line:step5}
\ENDIF
\ENDFOR
\STATE $\Omega:=\bigcup_{\gamma_0\in \mathcal A} G \cdot Q(\gamma_0)$\label{line:step6}
\STATE $\Omega :=$ set of minimal elements of $\Omega(k)$\label{line:step7}
\STATE Choose $w_0\in Q(\gamma)$ such that $\dim(\lambda_\Omega(w_0))=k$.
\STATE $\mathcal C := \{\lambda_\Omega(w_0)\}$
\STATE $\mathcal H := \{h_\Omega(\lambda_\Omega(w_0))\}$
\STATE $\mathcal{F}:=\{(\eta,v)\mid\eta\seite\lambda_\Omega(w_0) \text{ interior facet, } v\in\lambda_\Omega(w_0)^\vee \text{ its inner normal vector}\}$
\WHILE {there is $(\eta,v)\in\mathcal{F}$}\label{line:step12}
\label{fan trav}
\STATE Find $w\in Q(\gamma)$ such that $\eta\seite\lambda_\Omega(w)$ is a facet and $-v\in\lambda_\Omega(w)^\vee$.
\IF{$G \cdot h_\Omega(\lambda_\Omega(w))\cap \mathcal H =\emptyset$}\label{line:step14}
\STATE $\mathcal C := \mathcal C \cup \{\lambda_\Omega(w)\}$\label{line:step15}
\STATE $\mathcal H := \mathcal H \cup \{h_\Omega(\lambda_\Omega(w))\}$\label{line:step16}
\STATE $\mathcal F := \mathcal F \ominus \{(\tilde\eta,\tilde v)\mid\tilde \eta\seite\lambda_\Omega(w) \text{ interior facet, } \tilde v\in\lambda_\Omega(w)^\vee \text{ its inner normal vector}\}$
\ELSE
\STATE $\mathcal F := \mathcal F\setminus \{(\eta,v)\}$\label{line:step19}
\ENDIF
\ENDWHILE
\RETURN $\mathcal C$
\end{algorithmic}
\end{algo}

Examples for the use of Algorithm~\ref{algo:gitfanWithSymmetry} are given in Section~\ref{examples}.
We turn to the proof of Algorithm~\ref{algo:gitfanWithSymmetry}.
A first step is to show that the reduction of the set of orbit cones (see Line~\ref{line:step7}) and therefore also of the set of $\aa$-faces does not change the resulting GIT-fan, that is, we have to show that it suffices to consider the minimal orbit cones.

We call $Q(\gamma_{0})\in \Omega(k)$, where $\gamma_{0}\preceq\gamma$ is an $\mathfrak{a}$-face, a \emph{minimal}
orbit cone if for each full-dimensional cone $Q(\gamma_{1})\ne Q(\gamma_{0})$, where
$\gamma_{1}\preceq\gamma$ is an $\mathfrak{a}$-face, we have $Q(\gamma_{1})\not \subseteq Q(\gamma_{0})$.

\goodbreak

\begin{lemma}
\label{lem:minimal}
For the computation of the GIT-fan $\Lambda(\mathfrak{a},Q)$, it
suffices to consider the set $\Omega(k)_{\rm min}$ of minimal full-dimensional orbit cones, that is, given $w\in Q(\gamma)^{\circ}$, we have
\[
\lambda(w)\ =\ \bigcap_{\substack{\vartheta\in\Omega(k)_{\rm min},\\ w\in\vartheta}}\vartheta.
\]
\end{lemma}

\begin{proof} See~\cite{Ke} for the fact that $\Omega$ can be replaced by $\Omega(k)$ in the computation of $\lambda(w)$.
For the minimality, assume for some $w\in\gamma^{\circ}$, there was a cone $\tau\in\Omega(k)
\setminus\Omega(k)_{\rm min}$ with $w\in\tau^{\circ}$ such that
\[
\lambda^{+}\ :=\ \tau\cap\bigcap_{\vartheta\in\Omega(k)_{\rm min},\ w\in\vartheta}
\vartheta\ \subsetneq\ \bigcap_{\vartheta\in\Omega(k)_{\rm min},\ w\in\vartheta}
\vartheta\ =:\ \lambda_{0}.
\]
We may further assume that the GIT-cone $\lambda(w)$ is of full dimension and
that $\lambda^{+}=\lambda(w)$. Then there is a facet $\eta\preceq\tau$ with
$\eta^{\circ}\cap\lambda_{0}^{\circ}\ne\emptyset$.

Choosing a supporting hyperplane $H(\eta)$ for $\eta$ such that $\lambda^{+}\subseteq H(\eta)^{+}$, where by $H(\eta)^+$ we denote the positive halfspace defined by $H(\eta)$.
We see that there is
\[
w^{-}\ \in\ H(\eta)^{-}\cap\lambda_0^\circ\ \subseteq\ \lambda_{0}^{\circ
}\setminus\lambda^{+}.
\]
Since $\eta\in\Omega$ by \cite{BeHa} and the GIT-fan $\Lambda(\mathfrak{a},Q)$
is a fan constructed as the coarsest common refinement of all elements of
$\Omega$, the cone $\eta\in\Omega$ is a union of GIT-cones $\lambda
(w_{1}),\ldots,\lambda(w_{s})$ of codimension at least one. Since also
$\lambda(w^{-})$ must be a full-dimensional GIT-cone, there must be
$\tau^{\prime}\in\Omega$ with
\[
(\tau^{\prime})^{\circ}\cap\tau^{\circ}\ =\ \emptyset\qquad\text{and}%
\qquad\tau^{\prime}\cap\tau\ =\ \eta.
\]
We can choose $\tau^{\prime}$ minimal with this property and arrive at
$\tau^{\prime}\in\Omega(k)_{\rm min}$.
Then  $\lambda^{+}\subsetneq\lambda_{0}$ cannot be a subset, a contradiction.
\end{proof}

\begin{lemma}
 \label{lem:symmafaces}
 In the above setting, let $\gamma_0\seite \gamma$ be a face and let $\sigma\in G$.
 Then $\gamma_0$ is an $\aa$-face if and only if
 $\sigma(\gamma_0)$ is an $\aa$-face.
\end{lemma}

\begin{proof}
Write $z_{\gamma_0}$ for the $\gamma_0$-restriction
of $z := (1,\ldots,1)\in \KK^r$ as in Section~\ref{sec computing GIT-fans}. With $\mathbb{T} := (\KK^*)^r$, we have
 \begin{align*}
\gamma_{0}\text{ is an $\mathfrak{a}$-face} \quad &  \ \Longleftrightarrow
\quad V(\mathfrak{a})\cap(\mathbb{T}^{r}\cdot z_{\gamma_{0}})\ \neq\ \emptyset\\
&  \ \Longleftrightarrow\quad \sigma\left(  V(\mathfrak{a})\right)  \cap
\sigma(\mathbb{T}^{r}\cdot z_{\gamma_{0}})\ \neq\ \emptyset\\
&  \ \Longleftrightarrow\quad V(\mathfrak{a})\cap(\mathbb{T}^{r}\cdot z_{\sigma
(\gamma_{0})})\ \neq\ \emptyset\\
&  \ \Longleftrightarrow\quad \sigma(\gamma_{0})\text{ is an $\mathfrak{a}$-face.}\qedhere
\end{align*}
\end{proof}

\begin{proof}
[Proof of Algorithm~\ref{algo:gitfanWithSymmetry}]
Before we start with the proof of correctness of the output, note first that by Lemma~\ref{lem:symmafaces}, the set $G\cdot \mathcal A$, with $\mathcal A$ as constructed in Lines~\ref{line:step1} through~\ref{line:step5}, is indeed the set of $\mathfrak a$-faces.
Taking into account the induced action on the set of orbit cones, $\Omega$ as constructed in Step~\ref{line:step6} is indeed the set of orbit cones. Hence, by Lemma~\ref{lem:minimal}, restricting to the minimal orbit cones of maximal dimension in Step~\ref{line:step7} will not change the GIT-cones $\lambda_\Omega(w)$ computed in the remainder of the algorithm.

For correctness, we first show that $\mathcal C$ is a list of representatives for the orbits of the maximal cones of the GIT-fan, that is, we have $G\cdot \mathcal C=\Lambda(\mathfrak a, Q)(k)$.

For the inclusion ``$\subseteq$'',
  note that $\mathcal C\subseteq \Lambda(\mathfrak a,Q)(k)$ by correctness of Algorithm~\ref{algo:gitfan}.
  Moreover, given $\sigma\cdot \lambda_\Omega(w)\in G\cdot \mathcal C$ for some $\lambda_\Omega(w)\in \mathcal C$, we have
  \begingroup\allowdisplaybreaks
   \begin{eqnarray*}
    \sigma\cdot \lambda_\Omega(w)
    \,=\,
    A_{\sigma}\cdot\bigcap_{\substack{\theta\in\Omega, \\ w\in\theta}}\theta
    \,=\,
    A_\sigma\cdot\!\! \bigcap_{\substack{\theta\in\Omega, \\ A_\sigma^{-1}\cdot w\in\theta}}A_{\sigma}^{-1}\cdot\theta
    \,=\,
    \bigcap_{\substack{\theta\in\Omega, \\ A_\sigma^{-1}\cdot w\in\theta}}\theta
    \,=\,
    \lambda_\Omega(\sigma^{-1}\cdot w)
  \end{eqnarray*}
    \endgroup
  where the second equality holds because the $A_\sigma$ are linear isomorphisms permuting elements of $\Omega$, and the final inclusion again follows from the correctness of Algorithm \ref{algo:gitfan}.
    In particular, $\sigma\cdot \lambda_\Omega(w)$ is an element of $\Lambda(\aa,Q)$.

  We now prove the inclusion ``$\supseteq$''.
Consider $\lambda\in\Lambda(\mathfrak a,Q)(k)$. Let $\lambda_0$ denote the starting cone of Algorithm \ref{algo:gitfanWithSymmetry}. Define
  \[ d(\lambda)\,:=\,\min\left\{n\in\mathbb{N}\suchthat \begin{array}{c}\text{there are }\lambda_n:=\lambda,\lambda_{n-1},\ldots,\lambda_1\in\Lambda(\mathfrak a,Q)(k) \text{ such that} \\ \lambda_{i}\cap\lambda_{i-1} \text{ is a facet of both $\lambda_{i}$ and $\lambda_{i-1}$ for } i=1,\ldots,n\end{array}\right\}. \]
  Observe that such a chain of maximal GIT-cones always exists, so that $d(\lambda)$ is well-defined. We now do an induction on $d(\lambda)$ to prove that $\lambda\in G\cdot\mathcal C$, see Figure \ref{fig GIT group action}.

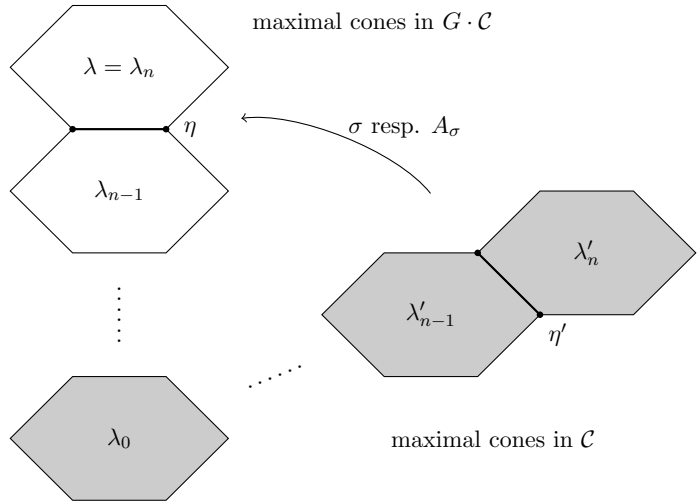
\begin{figure}[H]
\begin{center}
  \begin{tikzpicture}[scale=.82, transform shape]
    \fill[black!20] (-1.75,0) -- (-0.75,-1) -- (0.75,-1) -- (1.75,0) -- (0.75,1) -- (-0.75,1) -- cycle;
    \fill[black!20,xshift=5cm,yshift=2cm] (-1.75,0) -- (-0.75,-1) -- (0.75,-1) -- (1.75,0) -- (0.75,1) -- (-0.75,1) -- cycle;
    \fill[black!20,xshift=7.5cm,yshift=3cm] (-1.75,0) -- (-0.75,-1) -- (0.75,-1) -- (1.75,0) -- (0.75,1) -- (-0.75,1) -- cycle;

    \draw[yshift=6cm] (-1.75,0) -- (-0.75,-1) -- (0.75,-1) -- (1.75,0) -- (0.75,1) -- (-0.75,1) -- cycle;
    \draw[yshift=4cm] (-1.75,0) -- (-0.75,-1) -- (0.75,-1) -- (1.75,0) -- (0.75,1) -- (-0.75,1) -- cycle;
    \draw (-1.75,0) -- (-0.75,-1) -- (0.75,-1) -- (1.75,0) -- (0.75,1) -- (-0.75,1) -- cycle;
    \draw[xshift=5cm,yshift=2cm] (-1.75,0) -- (-0.75,-1) -- (0.75,-1) -- (1.75,0) -- (0.75,1) -- (-0.75,1) -- cycle;
    \draw[xshift=7.5cm,yshift=3cm] (-1.75,0) -- (-0.75,-1) -- (0.75,-1) -- (1.75,0) -- (0.75,1) -- (-0.75,1) -- cycle;

    \node[yshift=6cm] at (0,0) {$\lambda=\lambda_n$};
    \node[yshift=4cm] at (0,0) {$\lambda_{n-1}$};
    \node at (0,0) {$\lambda_0$};
    \node[xshift=5cm,yshift=2cm] at (0,0) {$\lambda_{n-1}'$};
    \node[xshift=7.5cm,yshift=3cm] at (0,0) {$\lambda_n'$};

    \draw[black,thick,yshift=4cm] (0.75,1) -- (-0.75,1);
    \fill[black,yshift=4cm] (0.75,1) circle (1.5pt);
    \fill[black,yshift=4cm] (-0.75,1) circle (1.5pt);
    \node[anchor=west,xshift=0.15cm,yshift=4cm] at (0.75,1) {$\eta$};

    \draw[black,thick,xshift=5cm,yshift=2cm] (1.75,0) -- (0.75,1);
    \fill[black,xshift=5cm,yshift=2cm] (1.75,0) circle (1.5pt);
    \fill[black,xshift=5cm,yshift=2cm] (0.75,1) circle (1.5pt);
    \node[anchor=north west,xshift=5cm,yshift=2cm] at (1.75,0) {$\eta'$};

    \path[->,shorten <=1cm, shorten >=1cm] (5.75,3) edge[bend right] node[anchor=south west] {$\sigma$ resp. $A_\sigma$} (0.75,5);
    \draw[draw opacity=0] (0,0) -- node[sloped] {$\ldots\ldots$} (0,4);
    \draw[draw opacity=0] (0,0) -- node[sloped] {$\ldots\ldots$} (5,2);

    \node at (6,0) {maximal cones in $\mathcal C$};
    \node[anchor=west] at (2,6.75) {maximal cones in $G\cdot \mathcal C$};
  \end{tikzpicture}
\end{center}
\caption{Group action on maximal GIT-cones.}
\label{fig GIT group action}
\end{figure}

  If $d(\lambda)=0$, then $\lambda=\lambda_0$ and $\lambda\in\mathcal C\subseteq G\cdot \mathcal C$ by construction. So suppose $n:=d(\lambda)>0$. Let $\lambda_n:=\lambda,\lambda_{n-1},\ldots,\lambda_1\in\Lambda(\mathfrak a,Q)(k)$ be such that $\lambda_{i}\cap\lambda_{i-1}$ is a facet of both for $i=1,\ldots,n$. By induction, $\lambda_{n-1}\in G\cdot\mathcal C$. This means that there exists a $\lambda_{n-1}'\in \mathcal C$ such that $\lambda_{n-1}=\sigma\cdot\lambda_{n-1}'$ for some $\sigma\in G$. Setting $\eta:=\lambda_n\cap\lambda_{n-1}$, the image $\eta':=\sigma^{-1}\cdot \eta$ is an interior facet of $\lambda_{n-1}'$ so that $(\eta',v')\in\mathcal F$ for a vector $v'\in(\lambda_{n-1}')^\vee$ at some step of the iteration.

  Take $\lambda_n'\in\Lambda(\mathfrak a,Q)(k)$ with $\lambda_{n-1}'\cap\lambda_n'=\eta'$.
  By Steps~\ref{line:step14} and~\ref{line:step15}, we then have $\theta\cdot \lambda_n' \in \mathcal C$ for some $\theta\in G$, possibly $\theta=e$. Hence, we obtain
  \[ \lambda_n = \sigma \cdot \lambda_n' \in G\cdot \mathcal C, \]
  as both sides of the equation are maximal cones of a polyhedral fan $\Lambda(\mathfrak a,Q)$ intersecting another maximal cone $\lambda_{n-1}$ in the same facet $\eta$.
Having shown $G\cdot \mathcal C=\Lambda(\mathfrak a, Q)(k)$, Steps~\ref{line:step14} and~\ref{line:step15} imply that $\mathcal C$ is a distinct system of representatives, finishing our proof for correctness.

For the termination, note that in each iteration of Steps~\ref{line:step12} through~\ref{line:step19} we either obtain a new GIT-cone $\lambda_\Omega(w)\in\mathcal C$, of which there are only finitely many, or the cardinality of the finite set $\mathcal F$ decreases by one. Hence the algorithm eventually terminates.
\end{proof}

We close this section with a series of remarks concerning the efficiency of Algorithm~\ref{algo:gitfanWithSymmetry} and sketching further improvements.

\begin{remark}
Instead of applying direct inclusion tests between orbit cones, Line~\ref{line:step7} can also be realized in a more efficient way by making use of the $G$-action:
 with $\aa$-faces $\gamma_i\seite \gamma$, we write
$$
G\cdot\gamma_{0} \,\sqsubseteq\, G\cdot\gamma_{1}
\quad
:\Longleftrightarrow
\quad
\text{for each $\gamma_{3}\in(G\cdot\gamma_{1})$ there is $\gamma_{2}\in(G\cdot\gamma_{0})$ with
 $\gamma_{2}\preceq\gamma_{3}$}.
$$
Defining
\begin{gather*}
\Omega_{1}
\  :=
\!\!\!\!
\bigcup_{\substack{G\cdot\gamma_{1}\text{ min.~w.r.t.~$\sqsubseteq$}%
\\\text{$\gamma_1$ $\aa$-face, $Q(\gamma_1)\in \Omega(k)$}}}
\!\!\!\!\!\!
Q(G\cdot\gamma_{1})
,\quad\qquad
\Omega_{2}\  :=\ \{\vartheta\in\Omega_{1}\mid\ \vartheta\text{ minimal w.r.t.}\,\subseteq\},
\end{gather*}
it then suffices to consider either one of the $\Omega_i$ instead of $\Omega$ in Line~\ref{line:step7} of Algorithm~\ref{algo:gitfanWithSymmetry} since $\Omega(k)_{\rm min}\subseteq \Omega_i$ for both~$i$. Hence, Lemma~\ref{lem:minimal} applies as well.
Note that $\Omega_1$ might be bigger than $\Omega(k)_{\rm min}$ but has the advantage that one can do the tests directly on the $\aa$-faces.
 \end{remark}

\begin{remark}
For the implementation of the algorithm it is not necessary to compute the
rays of the GIT-cones, we only use the descriptions in terms of half-spaces and hyperplanes.
\end{remark}

\begin{remark}
[Parallel computing]
The computations in the loop in Line~\ref{afacesloop1} are
independent, hence can be performed in parallel.  
A further improvement of the performance can be obtained by using a parallel approach to the fan-traversal.
\end{remark}

\begin{remark} 

An improvement of the memory usage can be achieved by the following strategy: Instead of listing the open facets in $\mathcal F$, we keep track of the maximal cones with open facets. For each such cone, we compute all its neighbouring cones in one iteration.
\end{remark}

\section{Examples\label{examples}}

In this section, we present two basic examples
for Algorithm~\ref{algo:gitfanWithSymmetry}
and explain
how they can be computed using our \textsc{Singular}-implementation~\cite{gitfanlib}.

\begin{example}
\label{ex:cube} Consider the polynomial ring $\mathbb{K}[T_{1},\ldots,T_{4}]$
with the $\mathbb{Z}^{2}$-grading $\deg(T_{j})=q_{j}$ given by the columns
\[
Q\ =\
\left[  q_{1},\ldots, q_{4}\right]  \ =\ \left[  \mbox{\footnotesize $
\begin{array}
[c]{rrrr}1 & -1 & -1 & 1\\
1 & 1 & -1 & -1
\end{array}
$}\right].
\]
Moreover, consider the principal ideal $\mathfrak{a}\subseteq \KT{4}$ generated by $g:=T_{1}T_{3}-T_{2}T_{4}$.
A~symmetry group $G$ for the graded algebra $\mathbb{K}%
[T_{1},\ldots,T_{4}]/\mathfrak{a}$ is then the symmetry group of the square%
\begin{center}
 \begin{minipage}{4cm}
\begin{align*}
G\,&=\,D_{4}\\
&=\,\left\langle (1,2)(3,4),(1,2,3,4)\right\rangle\\
&\leq\ S_{4}%
\end{align*}
 \end{minipage}
\qquad
\begin{minipage}{4cm}

\begin{center}
\begingroup
\begin{tikzpicture}[scale=1.2]
\coordinate (w1) at (1,1);
\coordinate (w2) at (-1,1);
\coordinate (w3) at  (-1,-1);
\coordinate (w4) at  (1,-1);
\coordinate (O) at  (0,0);
\fill[color=black!30] (w1)--(w2)--(w3)--(w4)--cycle;
\draw (w1)--(w2)--(w3)--(w4)--cycle;
\foreach \w in {1,...,4} {
\fill (w\w) circle (1.4pt);
}
\fill (O) circle (1.4pt) node[anchor=west]{$ $};
\draw (w1) node[anchor=west]{$q_1$};
\draw (w2) node[anchor=east]{$q_2$};
\draw (w3) node[anchor=east]{$q_3$};
\draw (w4) node[anchor=west]{$q_4$};
\draw[dashed] (0,-1.3) -- (0,1.3) node{\small };
\draw[dashed] (-1.3,0) -- (1.3,0) node{\small };
\draw[dashed] (-1.3,-1.3) -- (1.3,1.3) node{\small };
\draw[dashed] (-1.3,1.3) -- (1.3,-1.3) node{\small };
\pgfmathsetmacro{\ex}{0}
\pgfmathsetmacro{\ey}{0}
\draw[dashed,->,>=latex] (\ex,\ey) ++(45:.8) arc (45:135:.8);
\draw[dashed,->,>=latex] (\ex,\ey) ++(135:.8) arc (135:225:.8);
\draw[dashed,->,>=latex] (\ex,\ey) ++(225:.8) arc (225:315:.8);
\draw[dashed,->,>=latex] (\ex,\ey) ++(315:.8) arc (315:405:.8);
\end{tikzpicture}
\endgroup
\end{center}
 \end{minipage}
\end{center}

Write the canonical basis vectors $e_{1},e_{2}\in\mathbb{Z}^{2}$ as $e_{1}=-(q_{2}+q_{3})/2$ and $e_{2}=(q_{1}+q_{2})/2$.
The action of $G$ on $\mathbb{Q}^{2}$, in the sense of Definition~\ref{def:symmgrp}, is then given by
\[
A_{(1,2)(3,4)}\ =\ \left[  \mbox{\footnotesize $
\begin{array}{rr}
-1 & 0\\
0 & 1
\end{array}
$}\right]
,\qquad A_{(1,2,3,4)}\ =\ \left[  \mbox{\footnotesize $
\begin{array}{rr}
0 & -1\\
1 & 0
\end{array}
$}\right]  \ \in\ \mathrm{GL}(2,\mathbb{Z}).
\]
The action of $G$ decomposes the set of faces of the positive orthant
$\mathbb{Q}_{\geq0}^{4}$ into the disjoint union
\[
\{\gamma_{0}\}\ \cup\ (G\cdot\gamma_{1})\ \cup\ (G\cdot\gamma_{2}%
)\ \cup\ (G\cdot\gamma_{2}^{\prime})\ \cup\ (G\cdot\gamma_{3})\ \cup
\ \{\gamma_{4}\}
\]
where the cones $\gamma_i$, the size of their orbits, and the corresponding generators $g(T_{\gamma_i})$ in the sense of Section~\ref{sec computing GIT-fans} are as follows:
\[
\begin{tabular}
[c]{l|l|l}%
$\gamma$ & $\left\vert G\cdot\gamma\right\vert $ & $g(T_{\gamma})$\\\hline
$\gamma_{0}=\operatorname{cone}(0)$ & $1$ & $0$\\
$\gamma_{1}=\operatorname{cone}(e_{1})$ & $4$ & $0$\\
$\gamma_{2}=\operatorname{cone}(e_{1},e_{2})$ & $4$ & $0$\\
$\gamma_{2}^{\prime}=\operatorname{cone}(e_{1},e_{3})$ & $2$ & $T_{1}T_{3}$\\
$\gamma_{3}=\operatorname{cone}(e_{1},e_{2},e_{3})$ & $4$ & $T_{1}T_{3}$\\
$\gamma_{4}=\operatorname{cone}(e_{1},e_{2},e_{3},e_{4})$ & $1$ & $g$%
\end{tabular}
\ \
\ \
\]
Hence, the set of $\mathfrak{a}$-faces is given by the union
$\{\gamma_{0}\}\cup(G\cdot\gamma_{1})\cup(G\cdot\gamma_{2})\cup\{\gamma_{4}%
\}$.
Projecting the representatives of the respective orbits yields
\begin{gather*}
Q(\gamma_{0})\ =\ \mathrm{cone}(0),\qquad
Q(\gamma_{1})\ =\ \mathrm{cone}%
\left(  \mbox{\footnotesize $\left[\begin{array}{r}
1\\
1
\end{array}
\right]$}\right)  ,\\
Q(\gamma_{2})\ =\ \mathrm{cone}\left(  \mbox{\footnotesize $
	      \left[\begin{array}{r}
1\\
1
\end{array}
\right],\left[\begin{array}{r}
-1\\
1
\end{array}
\right] $}\right)  ,\qquad Q(\gamma_{4})\ =\ \mathbb{Q}^{2}.
\end{gather*}
We choose the weight vector $w_{0}:=(0,1)\in\mathbb{Z}^{2}$ and compute the
corresponding GIT-cone $\lambda(w_{0})=Q(\gamma_{2})$. By applying
$A_{(1,2,3,4)}$ successively, we obtain the remaining three maximal cones of
the GIT-fan $\Lambda(\mathfrak{a},Q)$ as depicted in the following figure:

\begin{center}
\begingroup
\begin{tikzpicture}[scale=1.2, transform shape,font=\footnotesize]
\coordinate (w1) at (1,1);
\coordinate (w2) at (-1,1);
\coordinate (w3) at  (-1,-1);
\coordinate (w4) at  (1,-1);
\coordinate (O) at  (0,0);
\fill[color=black!20] (w1)--(w2)--(w3)--(w4)--cycle;
\foreach \w in {1,...,4} {
\fill (w\w) circle (1.4pt);
}
\fill (O) circle (1.4pt) node[anchor=west]{{\tiny $(0,0)$ }};
\draw (w1) node[anchor=west]{$q_1$};
\draw (w2) node[anchor=east]{$q_2$};
\draw (w3) node[anchor=east]{$q_3$};
\draw (w4) node[anchor=west]{$q_4$};
\draw(w1) -- (w3);
\draw(w2) -- (w4);
\draw (0,.7) node{$\lambda(w_0)$};
\end{tikzpicture}
\endgroup
\end{center}

Using our implementation of Algorithm \ref{algo:gitfanWithSymmetry} in the
\textsc{Singular} library \textsc{gitfan.lib} we can compute the GIT-fan
up to symmetry using the command \texttt{GITfan(a, Q, G)}, where \texttt{a}, \texttt{Q} and \texttt{G}
stand for the ideal $\aa$, the matrix $Q$, and the symmetry group $G\subseteq S_r$, respectively.
\end{example}

As a second example, we compute the
Mori chamber decomposition of $\overline M_{0,5}$,
thereby reproducing results of  Arzhantsev/Hausen~\cite[Example~8.5]{AH}, Bernal~\cite{Ber}, and Dolgachev/Hu~\cite[3.3.24]{DolgachevHu} by making use of our symmetric GIT-fan algorithm.

\begin{example}
The Cox ring of $\overline M_{0,5}$ is isomorphic to the coordinate ring $R=\KK[T_1,\ldots,T_{10}]/\aa$ of the affine cone over the Grassmannian $\mathbb{G}(2,5)$ where the ideal $\aa$ is generated by the Pl\"ucker relations
\begin{center}
\begin{minipage}{3cm}
\begingroup\footnotesize
\begin{align*}
T_{5}T_{10}-T_{6}T_{9}+T_{7}T_{8},\\
  T_{1}T_{9}-T_{2}T_{7}+T_{4}T_{5},\\
  T_{1}T_{8}-T_{2}T_{6}+T_{3}T_{5},\\
  T_{1}T_{10}-T_{3}T_{7}+T_{4}T_{6},\\
  T_{2}T_{10}-T_{3}T_{9}+T_{4}T_{8}
\end{align*}
\endgroup
\end{minipage}
\qquad
\begin{minipage}{5cm}
 \[
 Q:=
 \left[
 \mbox{\footnotesize $
\begin{array}{rrrrrrrrrr}
  1 & 1 & 1 & 1 & 0 & 0 & 0 & 0 & 0 & 0\\
  1 & 0 & 0 & 0 & 1 & 1 & 1 & 0 & 0 & 0\\
  0 & 1 & 1 & 0 & 0 & 0 & -1 & 1 & 0 & 0\\
  0 & 1 & 0 & 1 & 0 & -1 & 0 & 0 & 1 & 0\\
  0 & 0 & 1 & 1 & -1 & 0 & 0 & 0 & 0 & 1
\end{array}$}
\right]
\]
\end{minipage}
\end{center}
and the $i$-th column of the matrix $Q$
is the degree $\deg(T_i)\in \ZZ^5$; this determines the $\ZZ^5$-grading of~$R$.
Using, e.g.,~\cite[Example~5.5]{HaKeWo}, we observe that there is an $S_5$-symmetry for the $H\cong (\KK^*)^5$-action on $V(\aa)$ where the symmetry group $S_5\cong G\subseteq S_{10}$ is generated by
\[
(2, 3)(5, 6)(9, 10),\qquad (1, 5, 9, 10, 3)(2, 7, 8, 4, 6)\quad \in\quad S_{10}.
\]
On the Cox ring,  $10$ of the $120$ elements of $G$ act by permutation of variables, whereas the remaining ones permute variables with a sign change.

We now apply Algorithm~\ref{algo:gitfanWithSymmetry} with input $\aa$, $Q$ and $G$ and obtain the following results:
By making use of the $S_5$-action, the number monomial containment tests via Algorithm~\ref{algo:isaface} can be reduced from $2^{10}=1024$ to $34$.
The set of $\aa$-faces consists of $172$ elements and
decomposes into $14$ orbits of lengths
\[
 1,\quad
 1,\quad
5,\quad
5,\quad
10,\quad
10,\quad
10,\quad
10,\quad
10,\quad
15,\quad
15,\quad
20,\quad
30,\quad
30.
\]
Projecting them via $Q$ yields the set $\Omega$ of $82$ orbit cones, amongst which $36$ are five-dimensional.
The set $\Omega(5)$ decomposes into the four $G$-orbits $G\cdot \vartheta_i$ with

\begin{gather*}
    \vartheta_1\,:=\,\cone\left[
    \mbox{\tiny $
    \begin{array}{rrrrrrrrrr}
    1 & 1 & 1 & 1 & 0 & 0 & 0 & 0 & 0 & 0 \\
    0 & 0 & 1 & 0 & 0 & 1 & 1 & 1 & 0 & 0 \\
    0 & 1 & 0 & 1 & 1 & 0 & 0 & -1 & 0 & 0 \\
    1 & 1 & 0 & 0 & 0 & -1 & 0 & 0 & 1 & 0 \\
    1 & 0 & 0 & 1 & 0 & 0 & -1 & 0 & 0 & 1
    \end{array}
    $}
\right],\\
    \vartheta_2\,:=\, \cone\left[
    \mbox{\tiny $
    \begin{array}{rrrrrrr}
    1 & 1 & 1 & 0 & 0 & 0 & 1 \\
    0 & 0 & 0 & 1 & 1 & 1 & 1 \\
    1 & 1 & 0 & -1 & 0 & 0 & 0 \\
    0 & 1 & 1 & 0 & 0 & -1 & 0 \\
    1 & 0 & 1 & 0 & -1 & 0 & 0
    \end{array}
    $}
    \right]
    ,\quad
    \vartheta_3\,:=\,\cone\left[
    \mbox{\tiny $
    \begin{array}{rrrrrrrrr}
    1 & 1 & 1 & 1 & 0 & 0 & 0 & 0 & 0 \\
    0 & 0 & 1 & 0 & 1 & 1 & 1 & 0 & 0 \\
    1 & 0 & 0 & 1 & 0 & -1 & 0 & 0 & 1 \\
    0 & 1 & 0 & 1 & 0 & 0 & -1 & 1 & 0 \\
    1 & 1 & 0 & 0 & -1 & 0 & 0 & 0 & 0
    \end{array}
    $}
    \right]
    ,\\
    \vartheta_4\,:=\,\cone\left[
    \mbox{\tiny $
    \begin{array}{rrrrrrrr}
    1 & 1 & 0 & 0 & 0 & 0 & 1 & 1 \\
    0 & 1 & 1 & 0 & 1 & 0 & 0 & 0 \\
    1 & 0 & -1 & 0 & 0 & 1 & 1 & 0 \\
    1 & 0 & 0 & 1 & -1 & 0 & 0 & 1 \\
    0 & 0 & 0 & 0 & 0 & 0 & 1 & 1
    \end{array}
    $}
    \right]
\end{gather*}
of respective lengths $1$, $10$, $10$,  and $15$.
Using Algorithm~\ref{algo:gitfanWithSymmetry}, we find that there are
six orbits $G\cdot \lambda_i$ of maximal GIT-cones with respective orbit lengths
\[
1,\qquad
5,\qquad
10,\qquad
10,\qquad
20,\qquad
30.
\]
This is in accordance with~\cite[Section 4.2]{Ber}.
Figure \ref{fig G25} shows the adjacency graph of the GIT-fan $\Lambda(\aa,Q)$, that is, the vertices represent the maximal cones and they are connected by an edge if and only if the corresponding GIT-cones share a common facet.
Different colors represent different orbits. Moreover, the figure shows the adjacency graph of the orbits.
Explicitly, the GIT-cones $\lambda_i$ representing the orbits are given as follows:
 \begin{gather*}
    \begin{array}{rclrcl}
       \lambda_1&:=& \cone\left[
    \mbox{\tiny $
    \begin{array}{rrrrrrrrrr}
    1 & 1 & 1 & 2 & 1 & 1 & 1 & 1 & 1 & 0 \\
    1 & 1 & 2 & 1 & 1 & 1 & 1 & 1 & 0 & 1 \\
    0 & 1 & 0 & 1 & 1 & 1 & 0 & 0 & 1 & 0 \\
    1 & 1 & 0 & 1 & 0 & 0 & 1 & 0 & 1 & 0 \\
    0 & 0 & 0 & 1 & 0 & 1 & 1 & 1 & 1 & 0
    \end{array}
    $}
    \right],
&
   \lambda_2&:=& \cone\left[
    \mbox{\tiny $
    \begin{array}{rrrrr}
    0 & 1 & 0 & 1 & 0 \\
    0 & 1 & 0 & 0 & 1 \\
    0 & 1 & 1 & 1 & 0 \\
    1 & 1 & 0 & 1 & 0 \\
    0 & 0 & 0 & 0 & -1
    \end{array}
    $}
    \right]
    ,
\\[8mm]
   \lambda_3
    &:=&
    \cone\left[
    \mbox{\tiny $
    \begin{array}{rrrrrr}
    1 & 1 & 1 & 1 & 0 & 0 \\
    1 & 2 & 1 & 1 & 1 & 1 \\
    1 & 0 & 1 & 0 & 0 & 0 \\
    0 & 0 & 1 & 1 & 0 & 0 \\
    0 & 0 & 0 & 0 & 0 & -1
    \end{array}
    $}
    \right]
,
&
    \lambda_4&:=& \cone\left[
    \mbox{\tiny $
    \begin{array}{rrrrr}
    0 & 0 & 0 & 1 & 0 \\
    0 & 1 & 1 & 1 & 1 \\
    0 & 0 & -1 & 0 & 0 \\
    1 & 0 & 0 & 1 & 0 \\
    0 & 0 & 0 & 0 & -1
    \end{array}
    $}
    \right]
  ,
\\[8mm]
  \lambda_5&:=&\cone\left[
    \mbox{\tiny $
    \begin{array}{rrrrr}
    0 & 0 & 0 & 1 & 0 \\
    0 & 1 & 0 & 1 & 1 \\
    0 & 0 & 1 & 1 & 0 \\
    1 & 0 & 0 & 1 & 0 \\
    0 & 0 & 0 & 0 & -1
    \end{array}
    $}
    \right]
,
&
      \lambda_6&:=& \cone\left[
    \mbox{\tiny $
    \begin{array}{rrrrr}
    0 & 0 & 1 & 1 & 0 \\
    0 & 1 & 1 & 1 & 1 \\
    0 & 0 & 1 & 0 & 0 \\
    1 & 0 & 1 & 1 & 0 \\
    0 & 0 & 0 & 0 & -1
    \end{array}
    $}
    \right].
        \end{array}
     \end{gather*}

 \begin{figure}[H]
\begin{center}
\includegraphics[width = 0.85\textwidth]{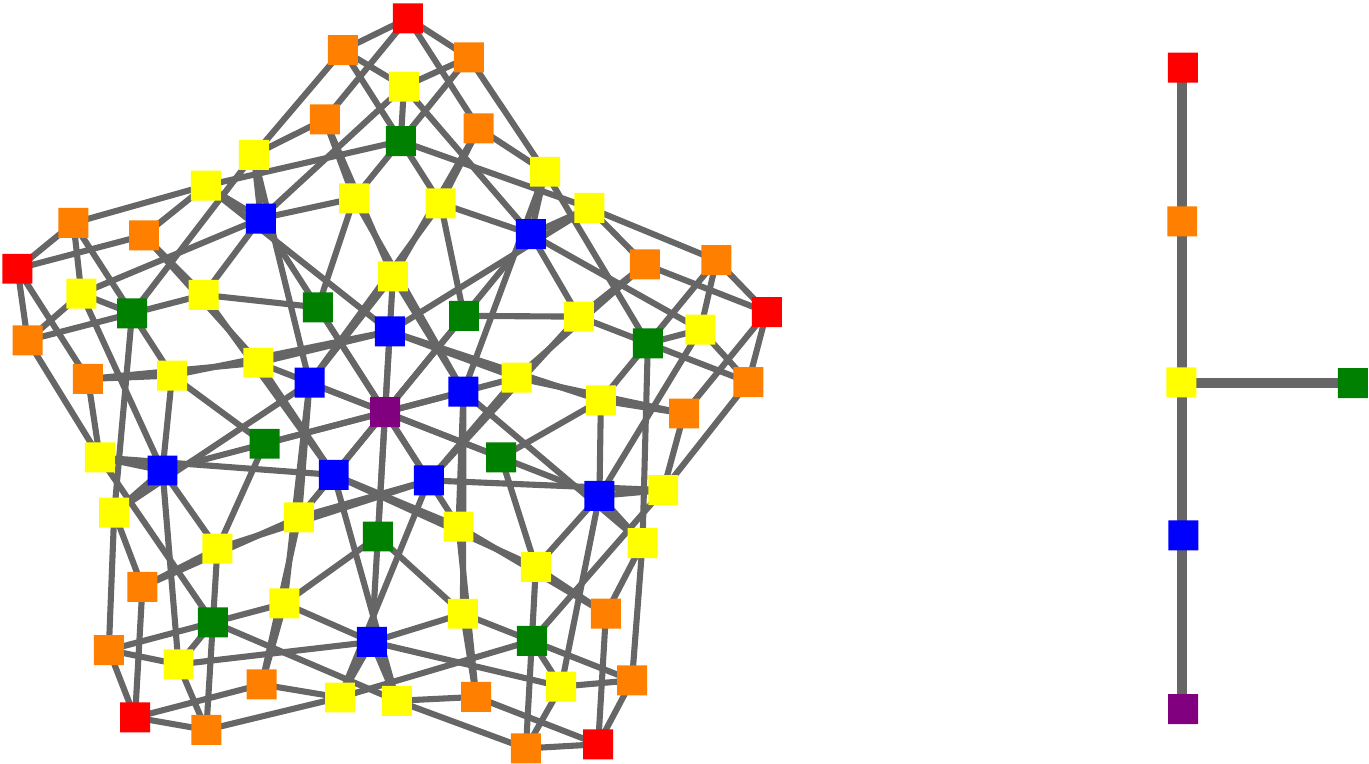}
\end{center}
 \caption{Adjacency graph of the maximal cones of the GIT-fan of $\mathbb{G}(2,5)$ and of their orbits under the $S_5$-action.}
 \label{fig G25}
 \end{figure}

\end{example}

\section{The Mori chamber decomposition of $\Mov(\overline{M}_{0,6})$\label{M06}}

In this section, we give a computational proof
of Theorem~\ref{thm:M06gitfan}, that is, we
determine the Mori chamber decomposition of the cone of movable divisor classes $\Mov(\overline{M}_{0,6})$ using Algorithm~\ref{algo:gitfanWithSymmetry}.
The input for the algorithm is the presentation
of the Cox ring of $\overline{M}_{0,6}$ in terms of generators
and relations determined by Bernal in~\cite[Theorem 5.4.1]{Ber}
together with the natural $S_6$-action thereon.
We summarize how to obtain this data in Construction~\ref{con:coxm06} and Algorithm~\ref{algo:coxm06}.

\begin{construction}[$S_6$-action on the polynomial ring, see \noexpand{\cite[Chapter 5.3]{Ber}}]
\label{con:coxm06}

Consider the effectively $\ZZ^{16}$-graded polynomial ring
$R$ with $40$ variables
\begin{align*}
R  &  :=\mathbb{K}[y_{1234},y_{1235},y_{1236},y_{1324},y_{1325},y_{1326}%
,y_{1423},y_{1425},y_{1426},y_{1523},y_{1524},y_{1526},y_{1623},\\
&  \phantom{= \mathbb{K}[}
\
y_{1624},y_{1625},
x_{12},x_{13},x_{14},x_{15},x_{16},x_{23}%
,x_{24},x_{25},x_{26},x_{34},x_{35},x_{36},x_{45},x_{46},x_{56},\\
&  \phantom{= \mathbb{K}
\
[}z_{123},z_{124},z_{125},z_{126},z_{134}%
,z_{135},z_{136},z_{145},z_{146},z_{156}]
\end{align*}
where the grading is given by providing the degrees of the generators $y_{abij}$, $x_{kl}$, $z_{mno}$ as columns of the integral $16\times 40$
matrix

\smallskip
\begin{center}
$Q :=$ \scalebox{0.7}{$\left[
\arraycolsep=1.4pt
\begin{array}
[c]{rrrrrrrrrrrrrrr}%
1 & 1 & 1 & 1 & 1 & 1 & 1 & 1 & 1 & 1 & 1 & 1 & 1 & 1 & 1\\
1 & 1 & 1 & 1 & 1 & 1 & 1 & 1 & 1 & 1 & 1 & 1 & 1 & 1 & 1\\
1 & 1 & 1 & 1 & 1 & 1 & 1 & 1 & 1 & 1 & 1 & 1 & 1 & 1 & 1\\
1 & 1 & 1 & 1 & 1 & 1 & 1 & 1 & 1 & 1 & 1 & 1 & 1 & 1 & 1\\
1 & 1 & 1 & 1 & 1 & 1 & 1 & 1 & 1 & 1 & 1 & 1 & 1 & 1 & 1\\
1 & 1 & 1 & 1 & 1 & 1 & 1 & 1 & 1 & 1 & 1 & 1 & 1 & 1 & 1\\
0 & 0 & 0 & 0 & 0 & 0 & 0 & -1 & -1 & 0 & -1 & -1 & 0 & -1 & -1\\
0 & 0 & 0 & 0 & -1 & -1 & 0 & 0 & 0 & -1 & 0 & -1 & -1 & 0 & -1\\
0 & 0 & 0 & -1 & 0 & -1 & -1 & 0 & -1 & 0 & 0 & 0 & -1 & -1 & 0\\
0 & 0 & 0 & -1 & -1 & 0 & -1 & -1 & 0 & -1 & -1 & 0 & 0 & 0 & 0\\
0 & -1 & -1 & 0 & 0 & 0 & 0 & 0 & 0 & -1 & -1 & 0 & -1 & -1 & 0\\
-1 & 0 & -1 & 0 & 0 & 0 & -1 & -1 & 0 & 0 & 0 & 0 & -1 & 0 & -1\\
-1 & -1 & 0 & 0 & 0 & 0 & -1 & 0 & -1 & -1 & 0 & -1 & 0 & 0 & 0\\
-1 & -1 & 0 & -1 & -1 & 0 & 0 & 0 & 0 & 0 & 0 & 0 & 0 & -1 & -1\\
-1 & 0 & -1 & -1 & 0 & -1 & 0 & 0 & 0 & 0 & -1 & -1 & 0 & 0 & 0\\
0 & -1 & -1 & 0 & -1 & -1 & 0 & -1 & -1 & 0 & 0 & 0 & 0 & 0 & 0
\end{array}
\begin{array}
[c]{|rrrrrrrrrrrrrrr}%
1 & 1 & 1 & 1 & 1 & 0 & 0 & 0 & 0 & 0 & 0 & 0 & 0 & 0 & 0\\
1 & 0 & 0 & 0 & 0 & 1 & 1 & 1 & 1 & 0 & 0 & 0 & 0 & 0 & 0\\
0 & 1 & 0 & 0 & 0 & 1 & 0 & 0 & 0 & 1 & 1 & 1 & 0 & 0 & 0\\
0 & 0 & 1 & 0 & 0 & 0 & 1 & 0 & 0 & 1 & 0 & 0 & 1 & 1 & 0\\
0 & 0 & 0 & 1 & 0 & 0 & 0 & 1 & 0 & 0 & 1 & 0 & 1 & 0 & 1\\
0 & 0 & 0 & 0 & 1 & 0 & 0 & 0 & 1 & 0 & 0 & 1 & 0 & 1 & 1\\
-1 & -1 & 0 & 0 & 0 & -1 & 0 & 0 & 0 & 0 & 0 & 0 & 0 & 0 & 0\\
-1 & 0 & -1 & 0 & 0 & 0 & -1 & 0 & 0 & 0 & 0 & 0 & 0 & 0 & 0\\
-1 & 0 & 0 & -1 & 0 & 0 & 0 & -1 & 0 & 0 & 0 & 0 & 0 & 0 & 0\\
-1 & 0 & 0 & 0 & -1 & 0 & 0 & 0 & -1 & 0 & 0 & 0 & 0 & 0 & 0\\
0 & -1 & -1 & 0 & 0 & 0 & 0 & 0 & 0 & -1 & 0 & 0 & 0 & 0 & 0\\
0 & -1 & 0 & -1 & 0 & 0 & 0 & 0 & 0 & 0 & -1 & 0 & 0 & 0 & 0\\
0 & -1 & 0 & 0 & -1 & 0 & 0 & 0 & 0 & 0 & 0 & -1 & 0 & 0 & 0\\
0 & 0 & -1 & -1 & 0 & 0 & 0 & 0 & 0 & 0 & 0 & 0 & -1 & 0 & 0\\
0 & 0 & -1 & 0 & -1 & 0 & 0 & 0 & 0 & 0 & 0 & 0 & 0 & -1 & 0\\
0 & 0 & 0 & -1 & -1 & 0 & 0 & 0 & 0 & 0 & 0 & 0 & 0 & 0 & -1
\end{array}
\begin{array}
[c]{|c}%
\\
\\
0_{6\times 10}\\
\\
\\
\\
\hline
\\
\\
\\
\\
\\
E_{10}\\
\\
\\
\\
\\
\end{array}
\right]$}
\end{center}
\smallskip
where we denote by $E_{10}$ the $10\times 10$ unit matrix
and by $0_{6\times 10}$ the $6\times 10$ zero matrix.
Moreover, consider the subgroup
$G \subseteq S_{40}$ isomorphic to $S_6$
generated by the permutations
\begin{center}
\scalebox{0.9}{\parbox{1\linewidth}{%
\begin{align*}
\sigma_{1}  &
=(4,7)(5,10)(6,13)(8,11)(9,14)(12,15)(17,21)(18,22)(19,23)(20,24)(35,40)(36,39)(37,38),\\
\sigma_{2}  &
=(1,4)(2,5)(3,6)(8,9)(11,12)(14,15)(16,17)(22,25)(23,26)(24,27)(32,35)(33,36)(34,37),\\
\sigma_{3}  &
=(2,3)(4,7)(5,8)(6,9)(10,11)(13,14)(17,18)(21,22)(26,28)(27,29)(31,32)(36,38)(37,39),\\
\sigma_{4}  &
=(1,2)(4,5)(7,10)(8,11)(9,12)(14,15)(18,19)(22,23)(25,26)(29,30)(32,33)(35,36)(39,40),\\
\sigma_{5}  &
=(2,3)(5,6)(8,9)(10,13)(11,14)(12,15)(19,20)(23,24)(26,27)(28,29)(33,34)(36,37)(38,39).
\end{align*}
}}
\end{center}
We then have an action of $G$ on $R$
$$
G\times R\,\to\,R,\qquad
\sigma_i\cdot T_j\, :=\, c_{\sigma_i,j}T_{\sigma(j)}
$$
where $T_j$ denotes the $j$-th variable of $R$
and the constants $c_{\sigma_i,j}$ are
the entries of  the following vectors $c_{\sigma_i}\in (\KK^*)^{40}$:
\[
\begin{tabular}
[c]{llll|l|ll}%
$c_{\sigma_{1}}$ & $=$ & $($ & $1^{7},-1^{2},1,-1,1^{2},-1,1,$ & $-1,1^{14},$
& $-1^{4},1^{6}$ & $)$,\\
$c_{\sigma_{2}}$ & $=$ & $($ & $1^{2},-1,1^{2},-1,1,-1^{2},1,-1^{2},1^{3},$ &
$1^{5},-1,1^{9},$ & $-1,1^{6},-1^{3}$ & $)$,\\
$c_{\sigma_{3}}$ & $=$ & $($ & $1^{3},-1^{6},1^{6},$ & $1^{9},-1,1^{5}$ &
$1^{2},-1^{3},1^{4},-1$ & $)$,\\
$c_{\sigma_{4}}$ & $=$ & $($ & $1^{13},-1^{2},$ & $1^{12},-1,1^{2},$ &
$-1,1^{2},-1,1^{2},-1^{2},1^{2}$ & $)$,\\
$c_{\sigma_{5}}$ & $=$ & $($ & $1,-1^{2},1^{8},-1,1^{2},-1,$ & $1^{14},-1,$ &
$-1^{2},1^{2},-1,1^{4},-1$ & $)$.
\end{tabular}
\]
Here, we write $a^b$ for the $b$-fold repetition of~$a$.

\end{construction}

From the data given by Construction~\ref{con:coxm06}, Algorithm \ref{algo:coxm06}
determines an explicit presentation
$R/I$ of the Cox ring $\Cox(\overline M_{0,6})$.

\begin{proposition}[Cox ring of $\overline M_{0,6}$]
\label{prop idealM06}
See \noexpand{\cite[Chapter 5.4]{Ber}}.
In the setting of Construction~\ref{con:coxm06},
the Cox ring of $\overline M_{0,6}$ is isomorphic to
$R/\aa$ where
$$
\aa\,:=\,
(I_1+G\cdot I_2):(y_{1234}\cdots z_{156})^{\infty},
$$
the $\ZZ^{16}$-degrees of the variables are the respective columns of the matrix~$Q$
and the ideals $I_1,I_2\subseteq R$ are defined as follows:
\begin{align*}
I_{1}\ :=\ \langle &x_{ij}x_{kl}z_{ijn}z_{kln}-x_{ik}x_{jl}z_{ikn}%
z_{jln}+x_{il}x_{jk}z_{iln}z_{jkn}\mid \,(i,j,k,l,m,n)\in\mathcal{M}\rangle,\\
 \mathcal{M}\ :=\ \{
 &\left(1,2,3,4,5,6\right),
\left(1,2,3,5,4,6\right),
\left(1,2,3,6,4,5\right),\\
 &\hspace{1cm}
 \left(1,2,4,5,3,6\right),
\left(1,2,4,6,3,5\right),
\left(1,2,5,6,3,4\right)\},\\
I_{2}\  \hphantom{:}=\ 
\langle &z_{126}y_{1423}-x_{13}x_{25}x_{46}z_{123}z_{134}%
z_{146}+x_{15}x_{24}x_{36}z_{124}z_{145}z_{156},\\
&  z_{126}y_{1425}+x_{13}x_{24}x_{56}z_{124}z_{134}z_{136}-x_{15}x_{23}%
x_{46}z_{125}z_{145}z_{146},\\
&  z_{125}y_{1426}-x_{13}x_{24}x_{56}z_{124}z_{134}z_{135}-x_{16}x_{23}%
x_{45}z_{126}z_{145}z_{146},\\
&  z_{126}y_{1523}-x_{13}x_{24}x_{56}z_{123}z_{135}z_{156}+x_{14}x_{25}%
x_{36}z_{125}z_{145}z_{146},\\
&  z_{126}y_{1524}+x_{13}x_{25}x_{46}z_{125}z_{135}z_{136}-x_{14}x_{23}%
x_{56}z_{124}z_{145}z_{156},\\
&  z_{124}y_{1526}-x_{13}x_{25}x_{46}z_{125}z_{134}z_{135}+x_{16}x_{23}%
x_{45}z_{126}z_{145}z_{156},\\
&  z_{125}y_{1623}+x_{13}x_{24}x_{56}z_{123}z_{136}z_{156}+x_{14}x_{26}%
x_{35}z_{126}z_{145}z_{146},\\
&  z_{125}y_{1624}+x_{13}x_{26}x_{45}z_{126}z_{135}z_{136}+x_{14}x_{23}%
x_{56}z_{124}z_{146}z_{156},\\
&  z_{135}y_{1625}+x_{12}x_{36}x_{45}z_{125}z_{126}z_{136}-x_{14}x_{23}%
x_{56}z_{134}z_{146}z_{156},\\
&  x_{12}y_{1234}+x_{13}x_{14}x_{25}x_{26}(z_{134})^{2}-x_{15}x_{16}%
x_{23}x_{24}(z_{156})^{2},\\
&  x_{12}y_{1235}+x_{13}x_{15}x_{24}x_{26}(z_{135})^{2}-x_{14}x_{16}%
x_{23}x_{25}(z_{146})^{2},\\
&  x_{12}y_{1236}-x_{13}x_{16}x_{24}x_{25}(z_{136})^{2}+x_{14}x_{15}%
x_{23}x_{26}(z_{145})^{2},\\
&  x_{13}y_{1324}+x_{12}x_{14}x_{35}x_{36}(z_{124})^{2}+x_{15}x_{16}%
x_{23}x_{34}(z_{156})^{2},\\
&  x_{13}y_{1325}+x_{12}x_{15}x_{34}x_{36}(z_{125})^{2}+x_{14}x_{16}%
x_{23}x_{35}(z_{146})^{2},\\
&  x_{13}y_{1326}+x_{12}x_{16}x_{34}x_{35}(z_{126})^{2}+x_{14}x_{15}%
x_{23}x_{36}(z_{145})^{2}\rangle.
\end{align*}
\end{proposition}

\begin{remark}
\label{algo:coxm06}
 In order to make the computation of generators for the ideal $\aa\subseteq R$ in Proposition~\ref{prop idealM06} feasible, Bernal~\cite{Ber} proposes to compute the saturation in two steps: 
 saturate the ideals $I_{1}$ and $I_2$ separately before 
 saturating their sum. 
\end{remark}

We can now directly use the results from the previous sections
to compute the Mori chamber decomposition of
$\overline M_{0,6}$.
To simplify the computation, we restrict to cones lying within the {\em moving cone}~$\operatorname{Mov}(\overline M_{0,6})$, i.e., the $16$-dimensional polyhedral cone
$$
\operatorname{Mov}(\overline M_{0,6})
\ =\
\bigcap_{i=1}^{40}
\operatorname{cone}(q_j\mid j\ne i)
\ \subseteq\
\operatorname{Eff}(\overline M_{0,6})
\ \subseteq\
\QQ^{16}
$$
where the $q_i\in\ZZ^{16}$ are the columns of the degree matrix $Q$ from Construction~\ref{con:coxm06}
and the cone $\operatorname{Eff}(\overline M_{0,6})$ of effective divisor classes equals
$\operatorname{cone}(q_1,\ldots,q_r)$.
The cone $\operatorname{Mov}(\overline M_{0,6})$ has $110$ facets and $128\,745$ rays.
It contains the cone $\operatorname{SAmple}(\overline M_{0,6})$ of semiample divisor classes.

\begin{remark}
 \label{rem:justific}
Recall from~\cite[Section~3.4]{ArDeHaLa} that the moving cone encodes the interesting part of Mori chamber decomposition in the following sense:
let $D$ be an effective divisor and let $\lambda$ be the GIT-cone of the Mori chamber decomposition  with $[D]\in\lambda^\circ$. Setting $X:=\overline M_{0,6}$, we obtain a birational map
$$
\varphi_D\colon X \to X(D)
\ :=\
\operatorname{Proj}\left(
\Gamma\left(X,\mathcal{A}(D)\right)\right)
,\qquad
\mathcal{A}(D)
\ :=\
\bigoplus_{n\in \ZZ_{\geq 0}}\mathcal{O}_X(nD).
$$
Then the map $\varphi_D$ is a {\em small quasimodification}, i.e., an isomorphism between open subsets of codimension at least two, if and only if $[D]\in \operatorname{Mov}(\overline M_{0,6})^\circ$,
a morphism if and only if $[D]\in\operatorname{SAmple}(\overline M_{0,6})$ and an isomorphism if and only if
$[D]\in \operatorname{Ample}(\overline M_{0,6})$.
\end{remark}

We are in the process
of investigating the feasibility of the computation of the full Mori chamber decomposition.

\begin{proof}[Computational proof of Theorem~\ref{thm:M06gitfan}]
This is an application of Algorithm~\ref{algo:gitfanWithSymmetry}:
as input we use the ideal of relations $\aa\subseteq\KK[y,x,z]$ of the Cox ring of $\overline M_{0,6}$ as given in Proposition~\ref{prop idealM06}
together with the corresponding grading matrix $Q$
as well as the symmetry group $G$ from Construction~\ref{con:coxm06}.
To restrict our computation to the cone of movable divisor classes $\sigma :=\operatorname{Mov}(\overline M_{0,6})$,
we change Algorithm~\ref{algo:gitfanWithSymmetry} slightly by redefining the notion of an {\em interior facet} to stand for facets $\eta\seite \lambda$ of GIT-cones $\lambda$ that meet $\sigma^\circ$ non-trivially.
This yields the Mori chamber decomposition of~$\sigma$.

A distinct set of
representatives of the maximal cones and the group action can be found in~\cite{gitfandata}. The numerical properties stated in the
theorem can easily be derived from this data by the corresponding functions provided in \textsc{gitfan.lib}.
\end{proof}

We immediately retrieve the following statement on the cone of semiample divisor classes; compare also~\cite[Section~6]{GiMac}.

\begin{corollary}
 The Mori cone of  $\overline M_{0,6}$ is the polyhedral cone in $\QQ^{16}$  generated by the $65$ rays in Table \ref{tab mori}. The semiample cone of  $\overline M_{0,6}$ (which is the dual of the Mori cone) has exactly $65$ facets and $3190$ rays.
   \end{corollary}

\begin{proof}By definition, the semiample cone is contained in the moving cone.
 By Theorem~\ref{thm:M06gitfan}, there is exactly one orbit of GIT-cones  of length one. Its unique element is, hence, the semiample cone.
\end{proof}

\begin{table}
\begin{scriptsize}
\begin{tabular}{c|c}
\setlength{\tabcolsep}{2.5pt}

\begin{tabular}{rrrrrrrrrrrrrrrr}
-2& 1&  0&  0&  1& -1& 1&  0&  0&  0&  0&  1&  1&  0&  0&  0\\
-2& 1&  0&  1&  0& -1& 1&  0&  0&  0&  1&  0&  1&  0&  0&  0\\
-2& 1&  1&  0&  0& -1& 1&  0&  0&  0&  1&  1&  0&  0&  0&  0\\
-2& 1&  1&  1&  1& -1& 0&  0&  0&  0&  0&  0&  0&  0&  0&  0\\
-1& 0&  0&  0&  0& -1& 1&  0&  0&  1&  0&  1&  0&  0&  0&  0\\
-1& 0&  0&  0&  0& -1& 1&  0&  0&  1&  1&  0&  0&  0&  0&  0\\
-1& 0&  0&  0&  0& -1& 1&  0&  1&  0&  0&  0&  1&  0&  0&  0\\
-1& 0&  0&  0&  0& -1& 1&  0&  1&  0&  1&  0&  0&  0&  0&  0\\
-1& 0&  0&  0&  0& -1& 1&  1&  0&  0&  0&  0&  1&  0&  0&  0\\
-1& 0&  0&  0&  0& -1& 1&  1&  0&  0&  0&  1&  0&  0&  0&  0\\
-1& 0&  0&  0&  1& -1& 1&  0&  1&  1&  0&  0&  0&  0&  0&  0\\
-1& 0&  0&  0&  1&  0& 0&  0&  0&  0&  0&  0&  0&  0&  0& -1\\
-1& 0&  0&  1&  0& -1& 1&  1&  0&  1&  0&  0&  0&  0&  0&  0\\
-1& 0&  0&  1&  0&  0& 0&  0&  0&  0&  0&  0&  0&  0& -1&  0\\
-1& 0&  1&  0&  0& -1& 1&  1&  1&  0&  0&  0&  0&  0&  0&  0\\
-1& 0&  1&  0&  0&  0& 0&  0&  0&  0&  0&  0&  0& -1&  0&  0\\
-1& 1&  0&  0&  0& -1& 1&  0&  0&  0&  0&  0&  1&  0&  0&  1\\
-1& 1&  0&  0&  0& -1& 1&  0&  0&  0&  0&  0&  1&  0&  1&  0\\
-1& 1&  0&  0&  0& -1& 1&  0&  0&  0&  0&  1&  0&  0&  0&  1\\
-1& 1&  0&  0&  0& -1& 1&  0&  0&  0&  0&  1&  0&  1&  0&  0\\
-1& 1&  0&  0&  0& -1& 1&  0&  0&  0&  1&  0&  0&  0&  1&  0\\
-1& 1&  0&  0&  0& -1& 1&  0&  0&  0&  1&  0&  0&  1&  0&  0\\
-1& 1&  0&  0&  1& -1& 1&  0&  0&  0&  0&  0&  0&  1&  1&  0\\
-1& 1&  0&  0&  1&  0& 0& -1&  0&  0&  0&  0&  0&  0&  0&  0\\
-1& 1&  0&  1&  0& -1& 1&  0&  0&  0&  0&  0&  0&  1&  0&  1\\
-1& 1&  0&  1&  0&  0& 0&  0& -1&  0&  0&  0&  0&  0&  0&  0\\
-1& 1&  1&  0&  0& -1& 1&  0&  0&  0&  0&  0&  0&  0&  1&  1\\
-1& 1&  1&  0&  0&  0& 0&  0&  0& -1&  0&  0&  0&  0&  0&  0\\
 0& 0& -1&  0&  0& -1& 1&  0&  0&  0&  0&  0&  1&  1&  0&  0\\
 0& 0& -1&  0&  0& -1& 1&  0&  0&  1&  0&  0&  0&  1&  0&  0\\
 0& 0& -1&  0&  0& -1& 1&  0&  0&  1&  0&  0&  1&  0&  0&  0\\
 0& 0& -1&  0&  0&  0& 1&  0&  0&  0&  0&  0&  0&  1&  0&  0\\
 0& 0& -1&  0&  0&  0& 1&  0&  0&  0&  0&  0&  1&  0&  0&  0\\
\end{tabular}

&
\setlength{\tabcolsep}{2.5pt}

\begin{tabular}{rrrrrrrrrrrrrrrr}
 0& 0& -1&  0&  0&  0& 1&  0&  0&  1&  0&  0&  0&  0&  0&  0\\
 0& 0&  0& -1&  0& -1& 1&  0&  0&  0&  0&  1&  0&  0&  1&  0\\
 0& 0&  0& -1&  0& -1& 1&  0&  1&  0&  0&  0&  0&  0&  1&  0\\
 0& 0&  0& -1&  0& -1& 1&  0&  1&  0&  0&  1&  0&  0&  0&  0\\
 0& 0&  0& -1&  0&  0& 1&  0&  0&  0&  0&  0&  0&  0&  1&  0\\
 0& 0&  0& -1&  0&  0& 1&  0&  0&  0&  0&  1&  0&  0&  0&  0\\
 0& 0&  0& -1&  0&  0& 1&  0&  1&  0&  0&  0&  0&  0&  0&  0\\
 0& 0&  0&  0& -1& -1& 1&  0&  0&  0&  1&  0&  0&  0&  0&  1\\
 0& 0&  0&  0& -1& -1& 1&  1&  0&  0&  0&  0&  0&  0&  0&  1\\
 0& 0&  0&  0& -1& -1& 1&  1&  0&  0&  1&  0&  0&  0&  0&  0\\
 0& 0&  0&  0& -1&  0& 1&  0&  0&  0&  0&  0&  0&  0&  0&  1\\
 0& 0&  0&  0& -1&  0& 1&  0&  0&  0&  1&  0&  0&  0&  0&  0\\
 0& 0&  0&  0& -1&  0& 1&  1&  0&  0&  0&  0&  0&  0&  0&  0\\
 0& 0&  0&  0&  0& -1& 0&  0&  0&  0&  0&  0&  0&  0&  0&  0\\
 0& 0&  0&  0&  0& -1& 1&  0&  0&  1&  0&  0&  0&  0&  0&  1\\
 0& 0&  0&  0&  0& -1& 1&  0&  0&  1&  0&  0&  0&  0&  1&  0\\
 0& 0&  0&  0&  0& -1& 1&  0&  1&  0&  0&  0&  0&  0&  0&  1\\
 0& 0&  0&  0&  0& -1& 1&  0&  1&  0&  0&  0&  0&  1&  0&  0\\
 0& 0&  0&  0&  0& -1& 1&  1&  0&  0&  0&  0&  0&  0&  1&  0\\
 0& 0&  0&  0&  0& -1& 1&  1&  0&  0&  0&  0&  0&  1&  0&  0\\
 0& 0&  0&  0&  0&  0& 0& -1&  0&  0&  0&  0&  0&  0&  0&  0\\
 0& 0&  0&  0&  0&  0& 0&  0& -1&  0&  0&  0&  0&  0&  0&  0\\
 0& 0&  0&  0&  0&  0& 0&  0&  0& -1&  0&  0&  0&  0&  0&  0\\
 0& 0&  0&  0&  0&  0& 0&  0&  0&  0& -1&  0&  0&  0&  0&  0\\
 0& 0&  0&  0&  0&  0& 0&  0&  0&  0&  0& -1&  0&  0&  0&  0\\
 0& 0&  0&  0&  0&  0& 0&  0&  0&  0&  0&  0& -1&  0&  0&  0\\
 0& 0&  0&  0&  0&  0& 0&  0&  0&  0&  0&  0&  0& -1&  0&  0\\
 0& 0&  0&  0&  0&  0& 0&  0&  0&  0&  0&  0&  0&  0& -1&  0\\
 0& 0&  0&  0&  0&  0& 0&  0&  0&  0&  0&  0&  0&  0&  0& -1\\
 0& 0&  0&  0&  1&  0& 0&  0&  0&  0& -1&  0&  0&  0&  0&  0\\
 0& 0&  0&  1&  0&  0& 0&  0&  0&  0&  0& -1&  0&  0&  0&  0\\
 0& 0&  1&  0&  0&  0& 0&  0&  0&  0&  0&  0& -1&  0&  0&  0\\
 \phantom{0}
\end{tabular}
\end{tabular}
\end{scriptsize}
\vspace{2mm}
\caption{Extremal rays of the Mori cone of $\overline M_{0,6}$, specified as rows.}
\label{tab mori}
\end{table}

\begin{remark}\label{4545}
The set of minimal orbit cones of dimension $16$ intersected with the moving cone is the union of two distinct
orbits consisting of $45$ elements each.
\end{remark}
\begin{remark}
As suggested by Diane Maclagan, one may expect that the restriction of the GIT-fan to $\operatorname{Mov}(\overline M_{0,6})$ can also be obtained by restricting to the subring \begin{align*}
 &  \mathbb{K}[
x_{12},x_{13},x_{14},x_{15},x_{16},x_{23}%
,x_{24},x_{25},x_{26},x_{34},x_{35},x_{36},x_{45},x_{46},x_{56},\\
&  \phantom{\mathbb{K}
\
[}z_{123},z_{124},z_{125},z_{126},z_{134}%
,z_{135},z_{136},z_{145},z_{146},z_{156}]
\end{align*}
of $R$, i.e., by 
 eliminating the variables corresponding to the Keel-Vermeire divisors from the ideal $\mathfrak{a}$ (constructed in Proposition~\ref{prop idealM06}). The corresponding computation shows that the set of minimal orbit cones of dimension $16$ intersected with the moving cone is the union of three distinct
orbits, two of length $45$, which agree with those mentioned in Remark~\ref{4545}, and one of length $15$.  Each cone in the orbit of length $15$ is the intersection of three cones in one of the orbits of length $45$, hence, the resulting Mori chamber decomposition of $\operatorname{Mov}(\overline M_{0,6})$ agrees with that in Theorem \ref{thm:M06gitfan}.
\end{remark}

\begin{remark}      
The 
computations for the proof of Theorem~\ref{thm:M06gitfan}
took approximately $8$ days, about one week for obtaining the $\mathfrak{a}$-faces (with a parallel computation on 16 cores) and one day for deriving the GIT-cones (by a parallel fan traversal on $16$ cores).
Making use of the group action of $G$, representing GIT-cones via the hash function of Construction~\ref{con:hash},
and applying Algorithm \ref{algo:sat} for the monomial containment tests turned out to be crucial for finishing
the computation.
\end{remark}

\bibliographystyle{abbrv}
\bibliography{gitfan}

\end{document}